\documentclass[a4paper, 11pt, dvipdfmx]{article}
\usepackage[top=30truemm, bottom=30truemm, left=20truemm, right=20truemm]{geometry}
\usepackage{amsthm}
\usepackage{amsmath}
\usepackage{amssymb}
\usepackage{amsfonts}
\usepackage{mathtools}
\usepackage{bm}
\usepackage{color}
\usepackage[hidelinks]{hyperref}
\newtheoremstyle{mystyle}
	{\baselineskip}
	{\baselineskip}
	{\itshape}
	{}
	{\bfseries}
	{.}
	{1em}
	{}

\theoremstyle{mystyle}
	\newtheorem{thm}{Theorem}[section]
	\newtheorem{lem}[thm]{Lemma}
	
	\newtheorem{prop}[thm]{Proposition}

\newtheoremstyle{mystyle2}
	{\baselineskip}
	{\baselineskip}
	{\normalfont}
	{}
	{\bfseries}
	{.}
	{1em}
	{}

\theoremstyle{mystyle2}
	\newtheorem{rem}[thm]{Remark}
	\newtheorem{definition}[thm]{Definition}
\newcommand{\C}{\mathbb{C}}
\newcommand{\R}{\mathbb{R}}
\newcommand{\Z}{\mathbb{Z}}
\renewcommand{\AA}{\mathcal{A}}
\newcommand{\MM}{\mathcal{M}}
\newcommand{\abs}[1]{\left\lvert #1 \right\rvert}
\newcommand{\norm}[1]{\left\lVert #1 \right\rVert}
\DeclareMathOperator{\re}{Re}
\DeclareMathOperator{\im}{Im}

\numberwithin{equation}{section}
\numberwithin{table}{section}
\numberwithin{figure}{section}
\allowdisplaybreaks[4]
\begin{document}
\title{
	Asymptotic behavior of global solutions to the complex Ginzburg--Landau type equation in the super Fujita-critical case
}

\author{
	Ryunosuke Kusaba$^{*1}$ and Tohru Ozawa$^{2}$
\bigskip \\
	$^{1}$Department of Pure and Applied Physics, \\
	Graduate School of Advanced Science and Engineering, \\
	Waseda University, 3-4-1 Okubo, Shinjuku-ku, Tokyo 169-8555, JAPAN
\bigskip \\
	$^{2}$Department of Applied Physics, \\
	Waseda University, 3-4-1 Okubo, Shinjuku-ku, Tokyo 169-8555, JAPAN
}

\date{}

\footnotetext[1]{$^{*}$Corresponding author: Ryunosuke Kusaba (e-mail: \texttt{ryu2411501@akane.waseda.jp})}

\maketitle
\begin{abstract}
	We present weighted estimates and higher order asymptotic expansions of global solutions to the complex Ginzburg--Landau (CGL) type equation in the super Fujita-critical case.
	Our approach is based on commutation relations between the CGL semigroup and monomial weights in $\R^{n}$ for the weighted estimates and on the Taylor expansions with respect to the both space and time variables for the asymptotic expansions.
	We also characterize the optimal decay rate in time of the remainder for the asymptotic expansion from the viewpoint of the moments of the initial data in space and those of the nonlinear term in spacetime.
\end{abstract}

\bigskip\noindent
\textbf{Keywords}: Complex Ginzburg--Landau equation, power-type nonlinearity, weighted estimate, large time behavior, asymptotic expansion.

\bigskip\noindent
\textbf{2020 Mathematics Subject Classification}: 35Q56, 35B40, 35C20.


\newpage
\section{Introduction}

We consider the asymptotic behavior of global solutions to the following Cauchy problem for the complex Ginzburg--Landau (CGL) type equation of the form
\begin{align*}
	\tag{P} \label{P}
	\begin{cases}
		\partial_{t} u- \nu \Delta u=f \left( u \right), &\qquad \left( t, x \right) \in \left( 0, + \infty \right) \times \R^{n}, \\
		u \left( 0 \right) =u_{0}, &\qquad x \in \R^{n},
	\end{cases}
\end{align*}
where $u \colon \left[ 0, + \infty \right) \times \R^{n} \to \C$ is an unknown function, $u_{0} \colon \R^{n} \to \C$ is a given data at $t=0$, $\nu \in \C$ with $\re \nu >0$ is a given parameter, and $f \colon \C \to \C$ is a nonlinear term satisfying $f \left( 0 \right) =0$ and
\begin{align*}
	\abs{f \left( \xi \right) -f \left( \eta \right)} \leq K \left( \abs{\xi}^{p-1} + \abs{\eta}^{p-1} \right) \abs{\xi - \eta}, \qquad \forall \xi, \eta \in \C
\end{align*}
for some $K>0$ and $p \in \left( 1, + \infty \right)$.
This nonlinearity is a natural extension of a single power of order $p$ such as
\begin{align*}
	f \left( \xi \right) = \lambda \xi^{p}, {\ }\lambda \abs{\xi}^{p}, {\ } \lambda \abs{\xi}^{p-1} \xi
\end{align*}
for some $\lambda \in \C \setminus \left\{ 0 \right\}$.

The complex Ginzburg--Landau equation
\begin{align}
	\label{eq:CGL}
	\partial_{t} u- \nu \Delta u= \mu u+ \lambda \abs{u}^{p-1} u
\end{align}
with $\mu \geq 0$ and $\re \lambda <0$ describes a wide variety of natural phenomena such as superconductivity and pattern formation (cf. \cite{Aranson-Kramer}).
It is also derived from the non-relativistic limit of a scalar field equation in homogeneous and isotropic spacetimes such as the de Sitter spacetime \cite{Nakamura}.
Therefore, the complex Ginzburg--Landau equation is said to be one of the most important equations in Physics.
From the viewpoint of mathematics, on one hand, if $\im \nu = \im \lambda = \mu =0$ and $u$ is a real-valued unknown function, then \eqref{eq:CGL} turns into the nonlinear heat equation, called the Fujita-type equation.
On the other hand, if $\re \nu = \re \lambda = \mu =0$, then \eqref{eq:CGL} becomes the nonlinear Schr\"{o}dinger equation.
Thus, the complex Ginzburg--Landau equation can be regarded as an intermediate equation describing the interaction between dissipation and dispersion.
In terms of both parabolic and dispersive equations, it has been studied in a large literature.
We refer the readers to \cite{Yang, Ginibre-Velo1996, Ginibre-Velo1997, Okazawa-Yokota2002, Okazawa, Yokota-Okazawa2006, Matsumoto-Tanaka2008, Yokota-Okazawa2008, Matsumoto-Tanaka2010, Okazawa-Yokota2010, Clement-Okazawa-Sobajima-Yokota, Shimotsuma-Yokota-Yoshii2014, Shimotsuma-Yokota-Yoshii2016} for the local and global well-posedness, \cite{Cazenave-Dickstein-Weissler, Cazenave-Dias-Figueira, Cazenave-Snoussi, Ikeda-Sobajima, Tomidokoro-Yokota} for the blowing-up of solutions in finite time, \cite{Machihara-Nakamura, Ogawa-Yokota} for the vanishing viscosity limit $\left( \re \nu \searrow 0 \right)$, and the references therein.
We notice that there has been a recent attention on complex-valued solutions to the Fujita-type equation \cite{Guo-Ninomiya-Shimojo-Yanagida, Chouichi-Otsmane-Tayachi, Nouaili-Zaag, Harada2016, Harada2017, Chouichi-Majdoub-Tayachi, Duong2019_1, Duong2019_2, Chen-Wang-Wang}.

In this paper, we are concerned with the complex Ginzburg--Landau ``type'' equation which is a special case of the parameter $\mu$ and a generalized case of the nonlinear term in \eqref{eq:CGL}.
In particular, we focus on the large time behavior of global solutions to \eqref{P}.
Similar to the case of the Fujita-type equation, the large time behavior of solutions to \eqref{P} changes depending on the size of the exponent $p$ relative to the exponent $1+2/n$, called the Fujita exponent.
In fact, the authors in \cite{Hayashi-Kaikina-Naumkin2003_1, Hayashi-Kaikina-Naumkin2003_2} showed that small amplitude solutions to \eqref{P} with $f \left( \xi \right) = \lambda \abs{\xi}^{p-1} \xi$ decay faster than the solution to the free case, namely, \eqref{P} with $f \equiv 0$, in the case where $p \leq 1+2/n$ (precisely, $p$ is equal to or sufficiently close to $1+2/n$) and
\begin{align}
	\label{eq:absorption}
	\re \left( \frac{\lambda \abs{\nu}^{n- \frac{n}{2} \left( p-1 \right)}}{\left( \left( p+1 \right) \lvert \nu \rvert^{2} + \left( p-1 \right) \nu^{2} \right)^{n/2}} \right) <0.
\end{align}
They also derived asymptotic profiles for the small amplitude solutions to \eqref{P} under the same condition.
Since \eqref{eq:absorption} is equivalent to $\lambda <0$ if $\im \nu = \im \lambda =0$, we can regard these results as extensions of those for the Fujita-type equation with absorption (see \cite{Gmira-Veron, Cazenave-Dickstein-Escobedo-Weissler} and the references therein).
However, due to the complex scalar field and the dispersion in \eqref{P}, the comparison principle, which is one of the most basic tools to study parabolic equations, is not available.
Therefore, we need to find the asymptotic profiles more directly.
For the case where $p>1+2/n$, the authors in \cite{Nakamura-Sato} showed that small amplitude solutions to \eqref{P} decay as fast as the free solution and asymptotically approach a constant multiple of the fundamental solution $G_{t \nu}$ to \eqref{P} with $f \equiv 0$.
However, they did not obtain the decay rate of the remainder and higher order asymptotic expansions.

The higher order asymptotic expansions for the semilinear heat equations were established in \cite{Ishige-Ishiwata-Kawakami, Ishige-Kawakami2012, Ishige-Kawakami2013, Ishige-Kawakami-Kobayashi}.
The authors effectively used the fact that the heat semigroup obtains $L^{1}$-decay if the initial data has vanishing moments up to some order.
Based on this fact, they decomposed the initial data and the nonlinear term according to the vanishing order of their moments, and then derived the higher order asymptotic expansions regarding the term with vanishing moments as the remainder and the other as the asymptotic profile.
This method is also available for other parabolic equations \cite{Kawakami, Ishige-Kawakami-Michihisa} and for the semilinear damped wave equation \cite{Kawakami-Ueda, Kawakami-Takeda}.
However, due to the decomposition, the asymptotic profiles seem to be too complicated to deduce the optimality for the decay rates of the remainders.
In particular, it may be impossible to derive estimates of the remainders from below.

Based on these points of view, we are naturally led to the purpose in two directions.
The first is to establish higher order asymptotic expansions of global solutions to \eqref{P} in the super Fujita-critical case: $p>1+2/n$.
The second is to derive the decay estimates in time of the remainders and their optimality.

Before we state our main results, we introduce some notation.
For each $q \in \left[ 1, + \infty \right]$, let $L^{q} \left( \R^{n} \right) =L^{q} \left( \R^{n}; \C \right)$ denote the standard Lebesgue space with the norm denoted by $\norm{\,\cdot\,}_{q}$.
We also employ the weighted $L^{1}$-space defined by
\begin{align*}
	L_{m}^{1} \left( \R^{n} \right) \coloneqq \left\{ \varphi \in L^{1} \left( \R^{n} \right); {\,} \text{$x^{\alpha} \varphi \in L^{1} \left( \R^{n} \right)$ for all $\alpha \in \Z_{\geq 0}^{n}$ with $\abs{\alpha} \leq m$} \right\}
\end{align*}
for each $m \in \Z_{>0}$, where $x^{\alpha} \varphi$ means the function $\R^{n} \ni x \mapsto x^{\alpha} \varphi \left( x \right) \in \C$.
We define the CGL semigroup $\left( e^{t \nu \Delta}; t \geq 0 \right)$ by
\begin{align*}
	e^{t \nu \Delta} \varphi \coloneqq \begin{cases}
		G_{t \nu} \ast \varphi, &\qquad t>0, \\
		\varphi, &\qquad t=0
	\end{cases}
\end{align*}
for $\varphi \in L^{q} \left( \R^{n} \right)$ with $q \in \left[ 1, + \infty \right]$, where $G_{t \nu} \colon \R^{n} \to \C$ is the fundamental solution to \eqref{P} with $f \equiv 0$ given by
\begin{align*}
	G_{t \nu} \left( x \right) = \left( 4 \pi t \nu \right)^{- \frac{n}{2}} \exp \left( - \frac{\abs{x}^{2}}{4 t \nu} \right), \qquad x \in \R^{n}
\end{align*}
and $\ast$ is the convolution in $\R^{n}$.

We first give the definition of global solutions to \eqref{P}.

\begin{definition}
	Let $u_{0} \in \left( L^{1} \cap L^{\infty} \right) \left( \R^{n} \right)$.
	A function
	\begin{align*}
		u \in X \coloneqq \left( C \cap L^{\infty} \right) \left( \left[ 0, + \infty \right); L^{1} \left( \R^{n} \right) \right) \cap \left( C \cap L^{\infty} \right) \left( \left( 0, + \infty \right); L^{\infty} \left( \R^{n} \right) \right)
	\end{align*}
	is said to be a global (mild) solution to \eqref{P} if the integral equation
	\begin{align*}
		\tag{I} \label{I}
		u \left( t \right) =e^{t \nu \Delta} u_{0} + \int_{0}^{t} e^{\left( t-s \right) \nu \Delta} f \left( u \left( s \right) \right) ds
	\end{align*}
	holds in $\left( L^{1} \cap L^{\infty} \right) \left( \R^{n} \right)$ for all $t>0$.
\end{definition}

In what follows, we basically assume that $p>1+2/n$ and $u_{0} \in \left( L^{1} \cap L^{\infty} \right) \left( \R^{n} \right)$.
We also suppose that there exists a global solution $u \in X$ to \eqref{P} satisfying
\begin{align}
	\label{eq:P_decay}
	\sup_{q \in \left[ 1, + \infty \right]} \sup_{t>0} t^{\frac{n}{2} \left( 1- \frac{1}{q} \right)} \norm{u \left( t \right)}_{q} <+ \infty.
\end{align}
The above estimate implies that the global solution decays as fast as the free solution.
Such a global solution actually exists at least for small initial data, which can be shown by the standard contraction argument for \eqref{I} (cf. \cite[Appendix A]{Kusaba-Ozawa}).
We emphasize that we assume not the smallness of the initial data but the decay in time of the global solution.
We also remark that in the case of the Fujita-type equation with absorption, we can construct a real-valued global solution satisfying \eqref{eq:P_decay} even for large initial data by virtue of the comparison principle (see \cite{Ishige-Kawakami2012} and the references therein).
In this paper, we often use the estimate
\begin{align}
	\label{eq:P_decay2}
	\sup_{q \in \left[ 1, + \infty \right]} \sup_{t>0} \left( 1+t \right)^{\frac{n}{2} \left( 1- \frac{1}{q} \right)} \norm{u \left( t \right)}_{q} <+ \infty,
\end{align}
which follows from $u \in L^{\infty} \left( 0, + \infty; \left( L^{1} \cap L^{\infty} \right) \left( \R^{n} \right) \right)$ and \eqref{eq:P_decay}.

The following weighted estimates of the global solution to \eqref{P} are important to derive asymptotic expansions.

\begin{thm} \label{th:P_weight}
	Let $p>1+2/n$ and let $m \in \Z_{>0}$.
	Let $u_{0} \in \left( L_{m}^{1} \cap L^{\infty} \right) \left( \R^{n} \right)$ and let $u \in X$ be a global solution to \eqref{P} satisfying \eqref{eq:P_decay}.
	Then, $u \in C \left( \left[ 0, + \infty \right); L_{m}^{1} \left( \R^{n} \right) \right)$, and moreover there exists $C>0$ such that the estimate
	\begin{align*}
		\sum_{\abs{\alpha} =m} \norm{x^{\alpha} u \left( t \right)}_{1} \leq C \left( 1+t^{\frac{m}{2}} \right)
	\end{align*}
	holds for all $t>0$.
\end{thm}

As we mentioned above, the comparison principle for parabolic equations is not available to show Theorem \ref{th:P_weight} due to the complex scalar field and the dispersion in \eqref{P}.
To overcome this difficulty, we establish commutator estimates between the CGL semigroup and monomial weights in $\R^{n}$ introduced in \cite{Kusaba-Ozawa} for the heat semigroup.
For details, see Theorems \ref{th:CGL_commutator} and \ref{th:CGL_commutator_esti} in Section \ref{sec:CGL_weight}.
This method enables us to show Theorem \ref{th:P_weight} by direct and explicit computations in the framework of the integral equation \eqref{I} without any approximation of the global solution.
Hence, we can discuss the well-posedness and a priori estimates for \eqref{P} independently.
Taking into account the fact that some papers constructed global solutions to \eqref{P} by some compactness argument and proved the uniqueness of solutions separately (see \cite{Ginibre-Velo1996, Clement-Okazawa-Sobajima-Yokota} and the references therein), we can ignore the uniqueness problem in the asymptotic analysis.

The next theorem provides asymptotic expansions of the global solution to \eqref{P}.

\begin{thm} \label{th:P_asymptotics_lim}
	Let $m \in \Z_{\geq 0}$ and let $p>1+ \left( m+2 \right) /n$.
	Let $u_{0} \in \left( L^{1}_{m} \cap L^{\infty} \right) \left( \R^{n} \right)$ and let $u \in X$ be a global solution to \eqref{P} satisfying \eqref{eq:P_decay}.
	Then,
	\begin{align*}
		\lim_{t \to + \infty} t^{\frac{n}{2} \left( 1- \frac{1}{q} \right) + \frac{m}{2}} \norm{u \left( t \right) - \AA_{m} \left( t \right)}_{q} =0
	\end{align*}
	holds for any $q \in \left[ 1, + \infty \right]$, where
	\begin{align*}
		\AA_{m} \left( t \right) &\coloneqq \Lambda_{0, m} \left( t; u_{0} \right) + \sum_{\abs{\gamma} \leq m/2} \frac{\left( - \nu \right)^{\abs{\gamma}}}{\gamma !} \Lambda_{2 \gamma, m-2 \abs{\gamma}} \left( t; \psi_{\abs{\gamma}} \right), \\
		\Lambda_{\alpha, N} \left( t; \varphi \right) &\coloneqq \left( -2 \right)^{- \abs{\alpha}} t^{- \frac{\abs{\alpha}}{2}} \sum_{k=0}^{N} 2^{-k} t^{- \frac{k}{2}} \sum_{\abs{\beta} =k} \MM_{\beta} \left( \varphi \right) \delta_{t} \left( \bm{h}_{\nu, \alpha + \beta} G_{\nu} \right), \\
		\MM_{\beta} \left( \varphi \right) &\coloneqq \frac{1}{\beta !} \int_{\R^{n}} y^{\beta} \varphi \left( y \right) dy, \\
		\bm{h}_{\nu, \alpha} \left( x \right) &\coloneqq \sum_{2 \beta \leq \alpha} \frac{\left( -1 \right)^{\abs{\beta}} \alpha !}{\beta ! \left( \alpha -2 \beta \right) !} \nu^{- \abs{\alpha - \beta}} x^{\alpha -2 \beta}, \\
		\psi_{k} &\coloneqq \int_{0}^{+ \infty} s^{k} f \left( u \left( s \right) \right) ds,
	\end{align*}
	and $\delta_{t}$ is the dilation acting on functions $\varphi$ on $\R^{n}$ as
	\begin{align*}
		\left( \delta_{t} \varphi \right) \left( x \right) =t^{- \frac{n}{2}} \varphi \left( t^{- \frac{1}{2}} x \right), \qquad x \in \R^{n}.
	\end{align*}
\end{thm}

Although the profile $\Lambda_{\alpha, N} \left( t; \varphi \right)$ seems to be complicated, it is just the $N$-th order asymptotic profile of $\partial^{\alpha} e^{t \nu \Delta} \varphi$.
In fact, under suitable assumptions, we have
\begin{align*}
	\lim_{t \to + \infty} t^{\frac{n}{2} \left( 1- \frac{1}{q} \right) + \frac{\abs{\alpha} +N}{2}} \norm{\partial^{\alpha} e^{t \nu \Delta} \varphi - \Lambda_{\alpha, N} \left( t; \varphi \right)}_{q} =0
\end{align*}
for any $q \in \left[ 1, + \infty \right]$.
This fact follows from the Taylor expansion of the fundamental solution $G_{t \nu}$ appearing in the integral representation of the CGL semigroup with respect to the space variable:
\begin{align*}
	G_{t \nu} \left( x-y \right) = \sum_{\abs{\alpha} \leq N} \frac{1}{\alpha !} \left( -y \right)^{\alpha} \left( \partial^{\alpha} G_{t \nu} \right) \left( x \right) + \text{Remainder}.
\end{align*}
For details, see Propositions \ref{pro:CGL_asymptotics} and \ref{pro:CGL_asymptotics_lim} in Section \ref{sec:CGL_asymptotics}.
The finiteness of the moments appearing in the definition of $\Lambda_{2 \gamma, m-2 \abs{\gamma}} \left( t; \psi_{\abs{\gamma}} \right)$ is guaranteed by the decay estimate \eqref{eq:P_decay2} and the weighted estimates given in Theorem \ref{th:P_weight}.
We note that $\bm{h}_{\nu, \alpha}$ is the Hermite polynomial on $\R^{n}$ with a complex coefficient arising naturally from $\partial^{\alpha} G_{t \nu}$.

The proof of Theorem \ref{th:P_asymptotics_lim} is based on the asymptotic expansions of the CGL semigroup and the Taylor expansion of the Duhamel term in the integral equation \eqref{I} with respect to the time variable:
\begin{align*}
	e^{\left( t-s \right) \nu \Delta} f \left( u \left( s \right) \right) &= \sum_{k=0}^{N} \frac{1}{k!} \left( -s \nu \Delta \right)^{k} e^{t \nu \Delta} f \left( u \left( s \right) \right) + \mathrm{Remainder} \\
	&= \sum_{\abs{\gamma} \leq N} \frac{\left( - \nu \right)^{\abs{\gamma}}}{\gamma !} s^{\abs{\gamma}} \partial^{2 \gamma} e^{t \nu \Delta} f \left( u \left( s \right) \right) + \mathrm{Remainder}.
\end{align*}
In summary, we derive the asymptotic expansions of the global solution to \eqref{P} satisfying \eqref{eq:P_decay} by only using the Taylor expansions with respect to the both space and time variables.

If we suppose $u_{0} \in L^{1}_{m+1} \left( \R^{n} \right)$ instead of $u_{0} \in L^{1}_{m} \left( \R^{n} \right)$ in Theorem \ref{th:P_asymptotics_lim}, then we can derive decay estimates of the remainder for the asymptotic expansion given in Theorem \ref{th:P_asymptotics_lim}.

\begin{thm} \label{th:P_asymptotics}
	Let $m \in \Z_{\geq 0}$ and let $p>1+ \left( m+2 \right) /n$.
	Let $u_{0} \in \left( L^{1}_{m+1} \cap L^{\infty} \right) \left( \R^{n} \right)$ and let $u \in X$ be a global solution to \eqref{P} satisfying \eqref{eq:P_decay}.
	Then, for any $q \in \left[ 1, + \infty \right]$, there exists $C>0$ such that the estimates
	\begin{align*}
		t^{\frac{n}{2} \left( 1- \frac{1}{q} \right) + \frac{m}{2}} \norm{u \left( t \right) - \AA_{m} \left( t \right)}_{q} &\leq \begin{dcases}
			Ct^{- \left( \sigma - \frac{m}{2} \right)} &\quad \text{if} \quad 1+ \frac{m+2}{n} <p<1+ \frac{m+3}{n}, \\
			Ct^{- \frac{1}{2}} \log \left( 2+t \right) &\quad \text{if} \quad p=1+ \frac{m+3}{n}, \\
			Ct^{- \frac{1}{2}} &\quad \text{if} \quad p>1+ \frac{m+3}{n}
		\end{dcases}
	\end{align*}
	hold for all $t>0$, where
	\begin{align*}
		\sigma \coloneqq \frac{n}{2} \left( p-1 \right) -1> \frac{m}{2}
	\end{align*}
	and $\AA_{m} \left( t \right)$ is the same as in Theorem \ref{th:P_asymptotics_lim}.
\end{thm}

By a simple computation, we see that the $m$-th order asymptotic profile $\AA_{m} \left( t \right)$ has the following stratification structure.
First of all, we have
\begin{align*}
	\AA_{k} \left( t \right) - \AA_{k-1} \left( t \right) =t^{- \frac{k}{2}} \delta_{t} \left( \AA_{k} \left( 1 \right) - \AA_{k-1} \left( 1 \right) \right)
\end{align*}
for any $t>0$ and $k \in \left\{ 0, \ldots, m \right\}$, where $\AA_{-1} \left( t \right) \equiv 0$ (see Lemma \ref{lem:P_Lambda} in Section \ref{sec:proof_optimal}).
In the above identity, the negative power $t^{-k/2}$ means the decay of the asymptotic amplitude, the dilation $\delta_{t}$ means the parabolic self-similarity of the asymptotic profile, and the function $\AA_{k} \left( 1 \right) - \AA_{k-1} \left( 1 \right)$ on $\R^{n}$ means the shape of the asymptotic profile, which is represented explicitly as
\begin{align*}
	\AA_{k} \left( 1 \right) - \AA_{k-1} \left( 1 \right) =2^{-k} \sum_{\abs{\alpha} =k} {\,} \Biggl( \MM_{\alpha} \left( u_{0} \right) + \sum_{\substack{\beta +2 \gamma = \alpha \\ \abs{\gamma} \leq k/2}} \frac{\left( - \nu \right)^{\abs{\gamma}}}{\gamma !} \MM_{\beta} \left( \psi_{\abs{\gamma}} \right) \Biggr) {\,} \bm{h}_{\nu, \alpha} G_{\nu}.
\end{align*}
We emphasize that the coefficients in the above formula are completely determined from the moments of the initial data in space and those of the nonlinear term in spacetime.

By virtue of this stratification structure, we have the following theorem.

\begin{thm} \label{th:P_asymptotics_optimal}
	Let $m \in \Z_{\geq 0}$ and let $p>1+ \left( m+3 \right) /n$.
	Let $u_{0} \in \left( L^{1}_{m+1} \cap L^{\infty} \right) \left( \R^{n} \right)$ and let $u \in X$ be a global solution to \eqref{P} satisfying \eqref{eq:P_decay}.
	Then,
	\begin{align}
		\label{eq:P_asymptotics_optimal}
		\lim_{t \to + \infty} t^{\frac{n}{2} \left( 1- \frac{1}{q} \right) + \frac{m}{2} + \frac{1}{2}} \norm{u \left( t \right) - \AA_{m} \left( t \right)}_{q} = \norm{\AA_{m+1} \left( 1 \right) - \AA_{m} \left( 1 \right)}_{q}
	\end{align}
	holds for any $q \in \left[ 1, + \infty \right]$.
	In particular, the following assertions are equivalent:
	\begin{itemize}
		\item[\rm (i)]
			$\AA_{m+1} \left( 1 \right) - \AA_{m} \left( 1 \right) \not\equiv 0$.
		\item[\rm (ii)]
			There exists $\alpha \in \Z_{\geq 0}^{n}$ with $\abs{\alpha} =m+1$ such that
			\begin{align*}
				\MM_{\alpha} \left( u_{0} \right) + \sum_{\substack{\beta +2 \gamma = \alpha \\ \abs{\gamma} \leq \left( m+1 \right) /2}} \frac{\left( - \nu \right)^{\abs{\gamma}}}{\gamma !} \MM_{\beta} \left( \psi_{\abs{\gamma}} \right) \neq 0.
			\end{align*}
		\item[\rm (iii)]
			For any $q \in \left[ 1, + \infty \right]$ and $\varepsilon \in \left( 0, 1 \right)$, there exists $t_{0} >0$ such that the estimates
			\begin{align*}
				&0<t^{- \frac{1}{2}} \left( 1- \varepsilon \right) \norm{\AA_{m+1} \left( 1 \right) - \AA_{m} \left( 1 \right)}_{q} \\
				&\hspace{2cm} \leq t^{\frac{n}{2} \left( 1- \frac{1}{q} \right) + \frac{m}{2}} \norm{u \left( t \right) - \AA_{m} \left( t \right)}_{q} \\
				&\hspace{4cm} \leq t^{- \frac{1}{2}} \left( 1+ \varepsilon \right) \norm{\AA_{m+1} \left( 1 \right) - \AA_{m} \left( 1 \right)}_{q}
			\end{align*}
			hold for all $t>t_{0}$.
	\end{itemize}
\end{thm}

Since the assertion (i) holds if the initial data is sufficiently small in some sense (see Proposition \ref{pro:P_small_optimal} in Appendix \ref{sec:appendix} for details), the equivalence $\text{(i)} \Leftrightarrow \text{(iii)}$ in the above theorem implies that the decay rate of the remainder for the $m$-th order asymptotic expansion given in Theorem \ref{th:P_asymptotics} is optimal in the case where $p>1+ \left( m+3 \right) /n$.
In particular, we give a characterization of this optimality from the viewpoint of the moments of the initial data and the nonlinear term.
We remark that the optimality of the decay rates given in Theorem \ref{th:P_asymptotics} with $1+ \left( m+2 \right) /n<p \leq 1+ \left( m+3 \right) /n$ is still open (see also \cite[Remark 1.7]{Ishige-Ishiwata-Kawakami}).

\begin{rem}
	After we obtained these results, we found two papers \cite{Nakamura-Takeda, Ishige-Kawakami2022} which discuss the asymptotic behavior of global solutions to the semilinear parabolic equations in a similar manner.
	We have some advantages over them as follows.
	\begin{itemize}
		\item
			In \cite{Nakamura-Takeda}, the authors derived weighted estimates and higher order asymptotic expansions of small amplitude solutions to the semilinear heat equation in the de Sitter spacetime.
			In their approach to the weighted estimates, we need to assume the smallness of $\norm{\left( 1+ \abs{x}^{m} \right) u_{0}}_{1} + \norm{u_{0}}_{\infty}$ since we employ the contraction argument in $L^{\infty} \left( 0, + \infty; \left( L^{1}_{m} \cap L^{\infty} \right) \left( \R^{n} \right) \right)$ equipped with a certain time-weighted norm (see \cite[Theorem 1.1]{Nakamura-Takeda}).
			By contrast, in our approach, we assume only the smallness of $\norm{u_{0}}_{1} + \norm{u_{0}}_{\infty}$ even if we construct a global solution to \eqref{P} satisfying \eqref{eq:P_decay} by the contraction argument.
			For the asymptotic expansions, although their computations are almost the same as ours, they did not derive the explicit formulae of the asymptotic profiles with the Hermite polynomials and the decay estimates of the remainders (see \cite[Theorem 1.2]{Nakamura-Takeda}).
			Moreover, they had no discussion on the optimality for the decay rates of the remainders in any case.
		\item
			In \cite{Ishige-Kawakami2022}, the authors obtained weighted estimates and higher order asymptotic expansions of global solutions to the semilinear fractional diffusion equation.
			Their method is also valid for the Fujita-type equation.
			However, since they used the comparison principle to derive the weighted estimates, it is not available for \eqref{P} (see \cite[Theorem 5.1 (a)]{Ishige-Kawakami2022}).
			For the asymptotic expansions, similarly to our approach, they employed the Taylor expansions of the fundamental solution with respect to the both space and time variables.
			In order not to impose the regularity on the nonlinear term, the coefficients in their asymptotic profiles depend on time and the parabolic self-similarity of the profiles is broken.
			Hence, it seems difficult to deduce the optimality for the decay rates of the remainders.
			Although they also obtained a result of the case where the coefficients are constant, which corresponds to Theorem \ref{th:P_asymptotics_lim}, they assume a stronger condition for the exponent $p$ than that in Theorem \ref{th:P_asymptotics_lim} for some reason (see \cite[Corollary 5.1]{Ishige-Kawakami2022}).
	\end{itemize}
\end{rem}

This paper is organized as follows.
In the next section, we introduce basic estimates and asymptotic expansions of the CGL semigroup.
In Sections \ref{sec:proof_weight}, \ref{sec:proof_asymptotics_lim}, \ref{sec:proof_asymptotics}, and \ref{sec:proof_optimal}, we give the proofs of Theorems \ref{th:P_weight}, \ref{th:P_asymptotics_lim}, \ref{th:P_asymptotics}, and \ref{th:P_asymptotics_optimal}, respectively.
\section{Basic properties of the CGL semigroup}

In this section, we introduce basic properties of the CGL semigroup which play crucial roles in the proofs of our main results.
First of all, we prepare some notation.
Let $\Z_{>0}$ be the set of positive integers and let $\Z_{\geq 0} \coloneqq \Z_{>0} \cup \left\{ 0 \right\}$.
For $\alpha = \left( \alpha_{1}, \ldots, \alpha_{n} \right) \in \Z_{\geq 0}^{n}$ and $x= \left( x_{1}, \ldots, x_{n} \right) \in \R^{n}$, we define
\begin{align*}
	\abs{\alpha} \coloneqq \sum_{j=1}^{n} \alpha_{j}, \qquad \alpha ! \coloneqq \prod_{j=1}^{n} \alpha_{j} !, \qquad x^{\alpha} \coloneqq \prod_{j=1}^{n} x_{j}^{\alpha_{j}}, \qquad \partial^{\alpha} = \partial_{x}^{\alpha} \coloneqq \prod_{j=1}^{n} \partial_{j}^{\alpha_{j}}, \qquad \partial_{j} \coloneqq \frac{\partial}{\partial x_{j}}.
\end{align*}
For $\alpha = \left( \alpha_{1}, \ldots, \alpha_{n} \right)$, $\beta = \left( \beta_{1}, \ldots, \beta_{n} \right) \in \Z_{\geq 0}^{n}$, $\alpha \leq \beta$ means that $\alpha_{j} \leq \beta_{j}$ holds for any $j \in \left\{ 1, \ldots, n \right\}$.
In addition, we define
\begin{align*}
	\binom{\alpha}{\beta} \coloneqq \begin{dcases}
		\frac{\alpha !}{\beta ! \left( \alpha - \beta \right) !} = \prod_{j=1}^{n} \frac{\alpha_{j} !}{\beta_{j} ! \left( \alpha_{j} - \beta_{j} \right) !}, &\qquad \beta \leq \alpha, \\
		0, &\qquad \text{otherwise}.
	\end{dcases}
\end{align*}

\subsection{Decay and smoothing estimates} \label{sec:CGL_esti}
To consider basic properties of the CGL semigroup, we start with the self-similarity of the fundamental solution $G_{t \nu}$ described as
\begin{align}
	\label{eq:Gauss_self}
	G_{t \nu} \left( x \right) =t^{- \frac{n}{2}} G_{\nu} \left( t^{- \frac{1}{2}} x \right).
\end{align}
Based on this structure, we define the dilation $\delta_{t}$ by
\begin{align*}
	\left( \delta_{t} \varphi \right) \left( x \right) =t^{- \frac{n}{2}} \varphi \left( t^{- \frac{1}{2}} x \right), \qquad \varphi \in L_{\mathrm{loc}}^{1} \left( \R^{n} \right), {\ } x \in \R^{n}
\end{align*}
for each $t>0$.
Then, the family of the dilations $\left( \delta_{t}; t>0 \right)$ has the following properties:
\begin{itemize}
	\item[(i)]
		$\delta_{t} \delta_{s} = \delta_{ts}$ for any $t, s>0$.
	\item[(ii)]
		$\norm{\delta_{t} \varphi}_{q} =t^{- \frac{n}{2} \left( 1- \frac{1}{q} \right)} \norm{\varphi}_{q}$ for any $t>0$, $q \in \left[ 1, + \infty \right]$, and $\varphi \in L^{q} \left( \R^{n} \right)$.
\end{itemize}
By using the dilation $\delta_{t}$, we can rewrite \eqref{eq:Gauss_self} as
\begin{align}
	\label{eq:Gauss_self_dilation}
	G_{t \nu} = \delta_{t} G_{\nu},
\end{align}
whence follows
\begin{align}
	\label{eq:CGL_derivative}
	\partial^{\alpha} e^{t \nu \Delta} \varphi = \left( \partial^{\alpha} G_{t \nu} \right) \ast \varphi = \left( \partial^{\alpha} \left( \delta_{t} G_{\nu} \right) \right) \ast \varphi =t^{- \frac{\abs{\alpha}}{2}} \left( \delta_{t} \left( \partial^{\alpha} G_{\nu} \right) \right) \ast \varphi
\end{align}
for any $t>0$ and $\alpha \in \Z_{\geq 0}^{n}$.
Applying Young's inequality to the above identity yields the following estimate for the CGL semigroup.

\begin{lem} \label{lem:CGL_Lp-Lq}
	Let $1 \leq q \leq p \leq + \infty$, $\alpha \in \Z_{\geq 0}^{n}$, and $\varphi \in L^{q} \left( \R^{n} \right)$.
	Then, $\partial^{\alpha} e^{t \nu \Delta} \varphi \in L^{p} \left( \R^{n} \right)$ and the estimate
	\begin{align*}
		\norm{\partial^{\alpha} e^{t \nu \Delta} \varphi}_{p} \leq t^{- \frac{n}{2} \left( \frac{1}{q} - \frac{1}{p} \right) - \frac{\abs{\alpha}}{2}} \norm{\partial^{\alpha} G_{\nu}}_{r} \norm{\varphi}_{q}
	\end{align*}
	holds for any $t>0$, where $r \in \left[ 1, + \infty \right]$ with $1/p+1=1/r+1/q$.
\end{lem}

\begin{rem}
Since $\re \nu >0$, we see that $\partial^{\alpha} G_{\nu} \in L^{r} \left( \R^{n} \right)$ for any $r \in \left[ 1, + \infty \right]$ and $\alpha \in \Z_{\geq 0}^{n}$.
See also \eqref{eq:Gauss_derivative} in the next subsection.
\end{rem}

\begin{proof}[Proof of Lemma \ref{lem:CGL_Lp-Lq}]
	By \eqref{eq:CGL_derivative} and Young's inequality, we have
	\begin{align*}
		\norm{\partial^{\alpha} e^{t \nu \Delta} \varphi}_{p}
		&=t^{- \frac{\abs{\alpha}}{2}} \norm{\left( \delta_{t} \left( \partial^{\alpha} G_{\nu} \right) \right) \ast \varphi}_{p} \\
		&\leq t^{- \frac{\abs{\alpha}}{2}} \norm{\delta_{t} \left( \partial^{\alpha} G_{\nu} \right)}_{r} \norm{\varphi}_{q} \\
		&=t^{- \frac{n}{2} \left( 1- \frac{1}{r} \right) - \frac{\abs{\alpha}}{2}} \norm{\partial^{\alpha} G_{\nu}}_{r} \norm{\varphi}_{q} \\
		&=t^{- \frac{n}{2} \left( \frac{1}{q} - \frac{1}{p} \right) - \frac{\abs{\alpha}}{2}} \norm{\partial^{\alpha} G_{\nu}}_{r} \norm{\varphi}_{q},
	\end{align*}
	which in turn implies $\partial^{\alpha} e^{t \nu \Delta} \varphi \in L^{p} \left( \R^{n} \right)$.
\end{proof}

\subsection{Asymptotic expansions} \label{sec:CGL_asymptotics}

We next consider asymptotic expansions of the CGL semigroup.
To begin with, we introduce the Hermite polynomials with a complex coefficient \cite{Ozawa, Kusaba-Ozawa}.
For each $k \in \Z_{\geq 0}$, we define the Hermite polynomial of order $k$ by
\begin{align*}
	H_{\nu, k} \left( x \right) \coloneqq \left( -1 \right)^{k} \exp \left( \frac{x^{2}}{\nu} \right) \left( \frac{d}{dx} \right)^{k} \exp \left( - \frac{x^{2}}{\nu} \right), \qquad x \in \R.
\end{align*}
We remark that if $\nu =1$, then $H_{1, k}$ is the usual Hermite polynomial of order $k$.
By a simple calculation, we obtain the following representation of $H_{\nu, k}$:
\begin{align}
	\label{eq:Hermite_sum}
	H_{\nu, k} \left( x \right) = \sum_{j=0}^{\left[ k /2 \right]} \frac{\left( -1 \right)^{j} k !}{j! \left( k -2j \right) !} \nu^{- \left( k - j \right)} \left( 2x \right)^{k -2j},
\end{align}
where $\left[ k /2 \right] \coloneqq \max \left\{ j \in \Z_{\geq 0}; {\,} j \leq k /2 \right\}$.
In addition, for each $\alpha = \left( \alpha_{1}, \cdots, \alpha_{n} \right) \in \Z_{\geq 0}^{n}$, we define the multi-variable Hermite polynomial of order $\alpha$ by $\bm{H}_{\nu, \alpha} =H_{\nu, \alpha_{1}} \otimes \cdots \otimes H_{\nu, \alpha_{n}}$, namely,
\begin{align*}
	\bm{H}_{\nu, \alpha} \left( x \right) \coloneqq \prod_{j=1}^{n} H_{\nu, \alpha_{j}} \left( x_{j} \right) = \left( -1 \right)^{\abs{\alpha}} \exp \left( \frac{\abs{x}^{2}}{\nu} \right) \partial^{\alpha} \exp \left( - \frac{\abs{x}^{2}}{\nu} \right), \qquad x= \left( x_{1}, \ldots, x_{n} \right) \in \R^{n}.
\end{align*}
Then, it follows from \eqref{eq:Hermite_sum} that
\begin{align*}
	\bm{H}_{\nu, \alpha} \left( x \right)
	&= \prod_{j=1}^{n} \sum_{\beta_{j} =0}^{\left[ \alpha_{j} /2 \right]} \frac{\left( -1 \right)^{\beta_{j}} \alpha_{j} !}{\beta_{j} ! \left( \alpha_{j} -2 \beta_{j} \right) !} \nu^{- \left( \alpha_{j} - \beta_{j} \right)} \left( 2x_{j} \right)^{\alpha_{j} -2 \beta_{j}} \nonumber \\
	&= \sum_{2 \beta \leq \alpha} \frac{\left( -1 \right)^{\abs{\beta}} \alpha !}{\beta ! \left( \alpha -2 \beta \right) !} \nu^{- \abs{\alpha - \beta}} \left( 2x \right)^{\alpha -2 \beta}.
\end{align*}
Using the multi-variable Hermite polynomial $\bm{H}_{\nu, \alpha}$, we have
\begin{align*}
	\partial^{\alpha} G_{\nu} \left( x \right)
	&= \left( 4 \pi \nu \right)^{- \frac{n}{2}} \partial^{\alpha} \exp \left( - \frac{1}{\nu} \abs{\frac{x}{2}}^{2} \right) \\
	&=2^{- \abs{\alpha}} \left( 4 \pi \nu \right)^{- \frac{n}{2}} \left[ \partial_{y}^{\alpha} \exp \left( - \frac{\abs{y}^{2}}{\nu} \right) \right]_{y= \frac{x}{2}} \\
	&=2^{- \abs{\alpha}} \left( -1 \right)^{- \abs{\alpha}} \left( 4 \pi \nu \right)^{- \frac{n}{2}} \exp \left( - \frac{1}{\nu} \abs{\frac{x}{2}}^{2} \right) \left[ \left( -1 \right)^{\abs{\alpha}} \exp \left( \frac{\abs{y}^{2}}{\nu} \right) \partial_{y}^{\alpha} \exp \left( - \frac{\abs{y}^{2}}{\nu} \right) \right]_{y= \frac{x}{2}} \\
	&= \left( -2 \right)^{- \abs{\alpha}} G_{\nu} \left( x \right) \bm{H}_{\nu, \alpha} \left( \frac{x}{2} \right).
\end{align*}
Thus, by setting
\begin{align*}
	\bm{h}_{\nu, \alpha} \left( x \right) \coloneqq \bm{H}_{\nu, \alpha} \left( \frac{x}{2} \right) = \sum_{2 \beta \leq \alpha} \frac{\left( -1 \right)^{\abs{\beta}} \alpha !}{\beta ! \left( \alpha -2 \beta \right) !} \nu^{- \abs{\alpha - \beta}} x^{\alpha -2 \beta}, \qquad x \in \R^{n},
\end{align*}
we can rewrite $\partial^{\alpha} G_{\nu}$ as
\begin{align}
	\label{eq:Gauss_derivative}
	\partial^{\alpha} G_{\nu} = \left( -2 \right)^{- \abs{\alpha}} \bm{h}_{\nu, \alpha} G_{\nu}.
\end{align}
Combining \eqref{eq:Gauss_self_dilation} and \eqref{eq:Gauss_derivative}, we obtain
\begin{align*}
	\partial^{\alpha} G_{t \nu} = \partial^{\alpha} \left( \delta_{t} G_{\nu} \right) =t^{- \frac{\abs{\alpha}}{2}} \delta_{t} \left( \partial^{\alpha} G_{\nu} \right) = \left( -2 \right)^{- \abs{\alpha}} t^{- \frac{\abs{\alpha}}{2}} \delta_{t} \left( \bm{h}_{\nu, \alpha} G_{\nu} \right)
\end{align*}
for any $t>0$.

We also use the translation $\tau_{h}$ by $h \in \R^{n}$ defined as
\begin{align*}
	\left( \tau_{h} \varphi \right) \left( x \right) \coloneqq \varphi \left( x-h \right), \qquad \varphi \in L^{1}_{\text{loc}} \left( \R^{n} \right), {\ } x \in \R^{n}.
\end{align*}
We easily see that $\norm{\tau_{h} \varphi}_{q} = \norm{\varphi}_{q}$ for any $q \in \left[ 1, + \infty \right]$ and $\varphi \in L^{q} \left( \R^{n} \right)$.

Now, we are ready to state asymptotic expansions of the CGL semigroup.

\begin{prop} \label{pro:CGL_asymptotics}
	Let $m \in \Z_{\geq 0}$, $\varphi \in L^{1}_{m+1} \left( \R^{n} \right)$, $q \in \left[ 1, + \infty \right]$, and $\alpha \in \Z_{\geq 0}^{n}$.
	Then, the estimate
	\begin{align*}
		t^{\frac{n}{2} \left( 1- \frac{1}{q} \right) + \frac{\abs{\alpha} +m}{2}} \norm{\partial^{\alpha} e^{t \nu \Delta} \varphi - \Lambda_{\alpha, m} \left( t; \varphi \right)}_{q} \leq 2^{- \left( \abs{\alpha} +m+1 \right)} t^{- \frac{1}{2}} \sum_{\abs{\beta} =m+1} \frac{1}{\beta !} \norm{\bm{h}_{\nu, \alpha + \beta} G_{\nu}}_{q} \lVert x^{\beta} \varphi \rVert_{1}
	\end{align*}
	holds for any $t>0$, where
	\begin{align*}
		\Lambda_{\alpha, m} \left( t; \varphi \right) &\coloneqq \left( -2 \right)^{- \abs{\alpha}} t^{- \frac{\abs{\alpha}}{2}} \sum_{k=0}^{m} 2^{-k} t^{- \frac{k}{2}} \sum_{\abs{\beta} =k} \MM_{\beta} \left( \varphi \right) \delta_{t} \left( \bm{h}_{\nu, \alpha + \beta} G_{\nu} \right), \\
		\MM_{\beta} \left( \varphi \right) &\coloneqq \frac{1}{\beta !} \int_{\R^{n}} y^{\beta} \varphi \left( y \right) dy.
	\end{align*}
\end{prop}

\begin{proof}
	Let $t>0$.
	By Taylor's theorem, we have
	\begin{align*}
		&\left( \partial^{\alpha} G_{t \nu} \right) \left( x-y \right) \\
		&\hspace{0.5cm} = \sum_{\abs{\beta} \leq m} \frac{1}{\beta !} \left( -y \right)^{\beta} ( \partial^{\alpha + \beta} G_{t \nu} ) \left( x \right) + \sum_{\abs{\beta} =m+1} \frac{m+1}{\beta !} \int_{0}^{1} \left( 1- \theta \right)^{m} \left( -y \right)^{\beta} ( \partial^{\alpha + \beta} G_{t \nu} ) \left( x- \theta y \right) d \theta \\
		&\hspace{0.5cm} = \left( -2 \right)^{- \abs{\alpha}} t^{- \frac{\abs{\alpha}}{2}} \sum_{k=0}^{m} 2^{-k} t^{- \frac{k}{2}} \sum_{\abs{\beta} =k} \frac{1}{\beta !} y^{\beta} \left( \delta_{t} \left( \bm{h}_{\nu, \alpha + \beta} G_{\nu} \right) \right) \left( x \right) \\
		&\hspace{1.5cm} + \left( -1 \right)^{- \abs{\alpha}} 2^{- \left( \abs{\alpha} +m+1 \right)} t^{- \frac{\abs{\alpha} +m+1}{2}} \sum_{\abs{\beta} =m+1} \frac{m+1}{\beta !} \int_{0}^{1} \left( 1- \theta \right)^{m} y^{\beta} \left( \tau_{\theta y} \left( \delta_{t} \left( \bm{h}_{\nu, \alpha + \beta} G_{\nu} \right) \right) \right) \left( x \right) d \theta.
	\end{align*}
	Hence, the remainder is represented as
	\begin{align*}
		&\left( \partial^{\alpha} e^{t \nu \Delta} \varphi - \Lambda_{\alpha, m} \left( t; \varphi \right) \right) \left( x \right) \\
		&\hspace{1cm} = \int_{\R^{n}} \left( \left( \partial^{\alpha} G_{t \nu} \right) \left( x-y \right) - \left( -2 \right)^{- \abs{\alpha}} t^{- \frac{\abs{\alpha}}{2}} \sum_{k=0}^{m} 2^{-k} t^{- \frac{k}{2}} \sum_{\abs{\beta} =k} \frac{1}{\beta !} y^{\beta} \left( \delta_{t} \left( \bm{h}_{\nu, \alpha + \beta} G_{\nu} \right) \right) \left( x \right) \right) \varphi \left( y \right) dy \\
		&\hspace{1cm} = \left( -1 \right)^{- \abs{\alpha}} 2^{- \left( \abs{\alpha} +m+1 \right)} t^{- \frac{\abs{\alpha} +m+1}{2}} \\
		&\hspace{2cm} \times \sum_{\abs{\beta} =m+1} \frac{m+1}{\beta !} \int_{\R^{n}} \int_{0}^{1} \left( 1- \theta \right)^{m} \left( \tau_{\theta y} \left( \delta_{t} \left( \bm{h}_{\nu, \alpha + \beta} G_{\nu} \right) \right) \right) \left( x \right) y^{\beta} \varphi \left( y \right) d \theta dy.
	\end{align*}
	Taking $L^{q}$-norm of the above identity yields
	\begin{align*}
		&\norm{\partial^{\alpha} e^{t \nu \Delta} \varphi - \Lambda_{\alpha, m} \left( t; \varphi \right)}_{q} \\
		&\hspace{1cm} \leq 2^{- \left( \abs{\alpha} +m+1 \right)} t^{- \frac{\abs{\alpha} +m+1}{2}} \sum_{\abs{\beta} =m+1} \frac{m+1}{\beta !} \int_{\R^{n}} \int_{0}^{1} \left( 1- \theta \right)^{m} \norm{\tau_{\theta y} \left( \delta_{t} \left( \bm{h}_{\nu, \alpha + \beta} G_{\nu} \right) \right)}_{q} \lvert y^{\beta} \varphi \left( y \right) \rvert {\,} d \theta dy \\
		&\hspace{1cm} =2^{- \left( \abs{\alpha} +m+1 \right)} t^{- \frac{\abs{\alpha} +m+1}{2}} \sum_{\abs{\beta} =m+1} \frac{m+1}{\beta !} \int_{\R^{n}} \int_{0}^{1} \left( 1- \theta \right)^{m} \norm{\delta_{t} \left( \bm{h}_{\nu, \alpha + \beta} G_{\nu} \right)}_{q} \lvert y^{\beta} \varphi \left( y \right) \rvert {\,} d \theta dy \\
		&\hspace{1cm} =2^{- \left( \abs{\alpha} +m+1 \right)} t^{- \frac{n}{2} \left( 1- \frac{1}{q} \right) - \frac{\abs{\alpha} +m+1}{2}} \sum_{\abs{\beta} =m+1} \frac{1}{\beta !} \int_{\R^{n}} \norm{\bm{h}_{\nu, \alpha + \beta} G_{\nu}}_{q} \lvert y^{\beta} \varphi \left( y \right) \rvert {\,} dy \\
		&\hspace{1cm} =2^{- \left( \abs{\alpha} +m+1 \right)} t^{- \frac{n}{2} \left( 1- \frac{1}{q} \right) - \frac{\abs{\alpha} +m+1}{2}} \sum_{\abs{\beta} =m+1} \frac{1}{\beta !} \norm{\bm{h}_{\nu, \alpha + \beta} G_{\nu}}_{q} \lVert x^{\beta} \varphi \rVert_{1}.
	\end{align*}
\end{proof}

In this paper, the profile $\Lambda_{\alpha, m} \left( t; \varphi \right)$, which comes from the $m$-th order Taylor expansion of $\partial^{\alpha} G_{t \nu}$, is called the $m$-th order asymptotic profile of $\partial^{\alpha} e^{t \nu \Delta} \varphi$.
For $\varphi \in L_{m}^{1} \left( \R^{n} \right)$, although we cannot apply Proposition \ref{pro:CGL_asymptotics} to obtain a decay estimate of the remainder $\partial^{\alpha} e^{t \nu \Delta} \varphi - \Lambda_{\alpha, m} \left( t; \varphi \right)$, we can deduce that the remainder vanishes faster than $t^{- \frac{n}{2} \left( 1- \frac{1}{q} \right) - \frac{\abs{\alpha} +m}{2}}$ in $L^{q} \left( \R^{n} \right)$ as $t \to + \infty$.

\begin{prop} \label{pro:CGL_asymptotics_lim}
	Let $m \in \Z_{\geq 0}$, $\varphi \in L_{m}^{1} \left( \R^{n} \right)$, $q \in \left[ 1, + \infty \right]$, and $\alpha \in \Z_{\geq 0}^{n}$.
	Then,
	\begin{align*}
		\lim_{t \to + \infty} t^{\frac{n}{2} \left( 1- \frac{1}{q} \right) + \frac{\abs{\alpha} +m}{2}} \norm{\partial^{\alpha} e^{t \nu \Delta} \varphi - \Lambda_{\alpha, m} \left( t; \varphi \right)}_{q} =0,
	\end{align*}
	where $\Lambda_{\alpha, m} \left( t; \varphi \right)$ is defined in Proposition \ref{pro:CGL_asymptotics}.
\end{prop}

\begin{proof}
	We first show the case where $m=0$.
	Let $C_{c} \left( \R^{n} \right)$ denote the set of continuous functions on $\R^{n}$ with compact support.
	Since $C_{c} \left( \R^{n} \right)$ is dense in $L^{1} \left( \R^{n} \right)$, for given $\varphi \in L^{1} \left( \R^{n} \right)$, there exists a sequence $\left( \varphi_{j}; j \in \Z_{>0} \right)$ in $C_{c} \left( \R^{n} \right)$ such that $\norm{\varphi_{j} - \varphi}_{1} \to 0$ as $j \to + \infty$.
	In particular, $\varphi_{j} \in L_{1}^{1} \left( \R^{n} \right)$ for any $j \in \Z_{>0}$.
	By Lemma \ref{lem:CGL_Lp-Lq} and Proposition \ref{pro:CGL_asymptotics}, we have
	\begin{align*}
		&t^{\frac{n}{2} \left( 1- \frac{1}{q} \right) + \frac{\abs{\alpha}}{2}} \norm{\partial^{\alpha} e^{t \nu \Delta} \varphi - \Lambda_{\alpha, 0} \left( t; \varphi \right)}_{q} \\
		&\hspace{1cm} \leq t^{\frac{n}{2} \left( 1- \frac{1}{q} \right) + \frac{\abs{\alpha}}{2}} \norm{\partial^{\alpha} e^{t \nu \Delta} \left( \varphi - \varphi_{j} \right)}_{q} +t^{\frac{n}{2} \left( 1- \frac{1}{q} \right) + \frac{\abs{\alpha}}{2}} \norm{\partial^{\alpha} e^{t \nu \Delta} \varphi_{j} - \Lambda_{\alpha, 0} \left( t; \varphi_{j} \right)}_{q} \\
		&\hspace{2cm} +t^{\frac{n}{2} \left( 1- \frac{1}{q} \right) + \frac{\abs{\alpha}}{2}} \norm{\Lambda_{\alpha, 0} \left( t; \varphi_{j} \right) - \Lambda_{\alpha, 0} \left( t; \varphi \right)}_{q} \\
		&\hspace{1cm} \leq \norm{\partial^{\alpha} G_{\nu}}_{q} \norm{\varphi - \varphi_{j}}_{1} +2^{- \left( \abs{\alpha} +1 \right)} t^{- \frac{1}{2}} \sum_{\abs{\beta} =1} \frac{1}{\beta !} \norm{\bm{h}_{\nu, \alpha + \beta} G_{\nu}}_{q} \lVert x^{\beta} \varphi_{j} \rVert_{1} \\
		&\hspace{2cm} +2^{- \abs{\alpha}} t^{\frac{n}{2} \left( 1- \frac{1}{q} \right)} \abs{\MM_{0} \left( \varphi_{j} \right) - \MM_{0} \left( \varphi \right)} \norm{\delta_{t} \left( \bm{h}_{\nu, \alpha} G_{\nu} \right)}_{q} \\
		&\hspace{1cm} \leq 2^{1- \abs{\alpha}} \norm{\bm{h}_{\nu, \alpha} G_{\nu}}_{q} \norm{\varphi - \varphi_{j}}_{1} +2^{- \left( \abs{\alpha} +1 \right)} t^{- \frac{1}{2}} \sum_{\abs{\beta} =1} \frac{1}{\beta !} \norm{\bm{h}_{\nu, \alpha + \beta} G_{\nu}}_{q} \lVert x^{\beta} \varphi_{j} \rVert_{1}
	\end{align*}
	for any $t>0$, whence follows
	\begin{align*}
		\limsup_{t \to + \infty} t^{\frac{n}{2} \left( 1- \frac{1}{q} \right) + \frac{\abs{\alpha}}{2}} \norm{\partial^{\alpha} e^{t \nu \Delta} \varphi - \Lambda_{\alpha, 0} \left( t; \varphi \right)}_{q} \leq 2^{1- \abs{\alpha}} \norm{\bm{h}_{\nu, \alpha} G_{\nu}}_{q} \norm{\varphi - \varphi_{j}}_{1}.
	\end{align*}
	Since $\norm{\varphi - \varphi_{j}}_{1} \to 0$ as $j \to + \infty$, we obtain
	\begin{align*}
		\limsup_{t \to + \infty} t^{\frac{n}{2} \left( 1- \frac{1}{q} \right) + \frac{\abs{\alpha}}{2}} \norm{\partial^{\alpha} e^{t \nu \Delta} \varphi - \Lambda_{\alpha, 0} \left( t; \varphi \right)}_{q} =0,
	\end{align*}
	which implies the desired result.

	We next consider the case where $m \geq 1$.
	Let $t>0$.
	Similarly to the proof of Proposition \ref{pro:CGL_asymptotics}, applying Taylor's theorem yields
	\begin{align*}
		&\left( \partial^{\alpha} G_{t \nu} \right) \left( x-y \right) \\
		&\hspace{1cm} = \left( -2 \right)^{- \abs{\alpha}} t^{- \frac{\abs{\alpha}}{2}} \sum_{k=0}^{m-1} 2^{-k} t^{- \frac{k}{2}} \sum_{\abs{\beta} =k} \frac{1}{\beta !} y^{\beta} \left( \delta_{t} \left( \bm{h}_{\nu, \alpha + \beta} G_{\nu} \right) \right) \left( x \right) \\
		&\hspace{2cm} + \left( -1 \right)^{- \abs{\alpha}} 2^{- \left( \abs{\alpha} +m \right)} t^{- \frac{\abs{\alpha} +m}{2}} \sum_{\abs{\beta} =m} \frac{m}{\beta !} \int_{0}^{1} \left( 1- \theta \right)^{m-1} y^{\beta} \left( \tau_{\theta y} \left( \delta_{t} \left( \bm{h}_{\nu, \alpha + \beta} G_{\nu} \right) \right) \right) \left( x \right) d \theta \\
		&\hspace{1cm} = \left( -2 \right)^{- \abs{\alpha}} t^{- \frac{\abs{\alpha}}{2}} \sum_{k=0}^{m-1} 2^{-k} t^{- \frac{k}{2}} \sum_{\abs{\beta} =k} \frac{1}{\beta !} y^{\beta} \left( \delta_{t} \left( \bm{h}_{\nu, \alpha + \beta} G_{\nu} \right) \right) \left( x \right) \\
		&\hspace{2cm} + \left( -1 \right)^{- \abs{\alpha}} 2^{- \left( \abs{\alpha} +m \right)} t^{- \frac{\abs{\alpha} +m}{2}} \sum_{\abs{\beta} =m} \frac{m}{\beta !} \int_{0}^{1} \left( 1- \theta \right)^{m-1} y^{\beta} \left( \delta_{t} \left( \bm{h}_{\nu, \alpha + \beta} G_{\nu} \right) \right) \left( x \right) d \theta \\
		&\hspace{2cm} + \left( -1 \right)^{- \abs{\alpha}} 2^{- \left( \abs{\alpha} +m \right)} t^{- \frac{\abs{\alpha} +m}{2}} \\
		&\hspace{4cm} \times \sum_{\abs{\beta} =m} \frac{m}{\beta !} \int_{0}^{1} \left( 1- \theta \right)^{m-1} y^{\beta} \left( \tau_{\theta y} \left( \delta_{t} \left( \bm{h}_{\nu, \alpha + \beta} G_{\nu} \right) \right) - \delta_{t} \left( \bm{h}_{\nu, \alpha + \beta} G_{\nu} \right) \right) \left( x \right) d \theta \\
		&\hspace{1cm} = \left( -2 \right)^{- \abs{\alpha}} t^{- \frac{\abs{\alpha}}{2}} \sum_{k=0}^{m} 2^{-k} t^{- \frac{k}{2}} \sum_{\abs{\beta} =k} \frac{1}{\beta !} y^{\beta} \left( \delta_{t} \left( \bm{h}_{\nu, \alpha + \beta} G_{\nu} \right) \right) \left( x \right) \\
		&\hspace{2cm} + \left( -1 \right)^{- \abs{\alpha}} 2^{- \left( \abs{\alpha} +m \right)} t^{- \frac{\abs{\alpha} +m}{2}} \\
		&\hspace{4cm} \times \sum_{\abs{\beta} =m} \frac{m}{\beta !} \int_{0}^{1} \left( 1- \theta \right)^{m-1} y^{\beta} \left( \delta_{t} \left( \tau_{t^{-1/2} \theta y} \left( \bm{h}_{\nu, \alpha + \beta} G_{\nu} \right) - \bm{h}_{\nu, \alpha + \beta} G_{\nu} \right) \right) \left( x \right) d \theta.
	\end{align*}
	Thus, the remainder is represented as
	\begin{align*}
		&\left( \partial^{\alpha} e^{t \nu \Delta} \varphi - \Lambda_{\alpha, m} \left( t; \varphi \right) \right) \left( x \right) \\
		&\hspace{1cm} = \int_{\R^{n}} \left( \left( \partial^{\alpha} G_{t \nu} \right) \left( x-y \right) - \left( -2 \right)^{- \abs{\alpha}} t^{- \frac{\abs{\alpha}}{2}} \sum_{k=0}^{m} 2^{-k} t^{- \frac{k}{2}} \sum_{\abs{\beta} =k} \frac{1}{\beta !} y^{\beta} \left( \delta_{t} \left( \bm{h}_{\nu, \alpha + \beta} G_{\nu} \right) \right) \left( x \right) \right) \varphi \left( y \right) dy \\
		&\hspace{1cm} = \left( -1 \right)^{- \abs{\alpha}} 2^{- \left( \abs{\alpha} +m \right)} t^{- \frac{\abs{\alpha} +m}{2}} \\
		&\hspace{2cm} \times \sum_{\abs{\beta} =m} \frac{m}{\beta !} \int_{\R^{n}} \int_{0}^{1} \left( 1- \theta \right)^{m-1} \left( \delta_{t} \left( \tau_{t^{-1/2} \theta y} \left( \bm{h}_{\nu, \alpha + \beta} G_{\nu} \right) - \bm{h}_{\nu, \alpha + \beta} G_{\nu} \right) \right) \left( x \right) y^{\beta} \varphi \left( y \right) d \theta dy,
		\end{align*}
		whence follows
		\begin{align*}
		&t^{\frac{n}{2} \left( 1- \frac{1}{q} \right) + \frac{\abs{\alpha} +m}{2}} \norm{\partial^{\alpha} e^{t \nu \Delta} \varphi - \Lambda_{\alpha, m} \left( t; \varphi \right)}_{q} \\
		&\hspace{1cm} \leq 2^{- \left( \abs{\alpha} +m \right)} t^{\frac{n}{2} \left( 1- \frac{1}{q} \right)} \\
		&\hspace{2cm} \times \sum_{\abs{\beta} =m} \frac{m}{\beta !} \int_{\R^{n}} \int_{0}^{1} \left( 1- \theta \right)^{m-1} \norm{\delta_{t} \left( \tau_{t^{-1/2} \theta y} \left( \bm{h}_{\nu, \alpha + \beta} G_{\nu} \right) - \bm{h}_{\nu, \alpha + \beta} G_{\nu} \right)}_{q} \lvert y^{\beta} \varphi \left( y \right) \rvert {\,} d \theta dy \\
		&\hspace{1cm} \leq 2^{- \left( \abs{\alpha} +m \right)} \sum_{\abs{\beta} =m} \frac{m}{\beta !} \int_{\R^{n}} \int_{0}^{1} \norm{\tau_{t^{-1/2} \theta y} \left( \bm{h}_{\nu, \alpha + \beta} G_{\nu} \right) - \bm{h}_{\nu, \alpha + \beta} G_{\nu}}_{q} \lvert y^{\beta} \varphi \left( y \right) \rvert {\,} d \theta dy.
	\end{align*}
	Now, we define a function $\zeta_{q, \beta} \left( {\,\cdot\,}; t \right) \colon \left[ 0, 1 \right] \times \R^{n} \to \R$ by
	\begin{align*}
		\zeta_{q, \beta} \left( \theta, y; t \right) \coloneqq \norm{\tau_{t^{-1/2} \theta y} \left( \bm{h}_{\nu, \alpha + \beta} G_{\nu} \right) - \bm{h}_{\nu, \alpha + \beta} G_{\nu}}_{q}, \qquad \left( \theta, y \right) \in \left[ 0, 1 \right] \times \R^{n}
	\end{align*}
	for each $t>0$.
	Then, we have
	\begin{align*}
		\sup_{t>0} \sup_{0 \leq \theta \leq 1} \sup_{y \in \R^{n}} \zeta_{q, \beta} \left( \theta, y; t \right) &\leq 2 \norm{\bm{h}_{\nu, \alpha + \beta} G_{\nu}}_{q} <+ \infty.
	\end{align*}
	In addition, from the facts that the translation $\tau_{t^{-1/2} \theta y}$ is continuous in $L^{q} \left( \R^{n} \right)$ if $1 \leq q<+ \infty$ or that the function $\bm{h}_{\nu, \alpha + \beta} G_{\nu}$ is uniformly continuous on $\R^{n}$ if $q=+ \infty$, it follows that
	\begin{align*}
		\lim_{t \to + \infty} \zeta_{q, \beta} \left( \theta, y; t \right) =0
	\end{align*}
	for any $\left( \theta, y \right) \in \left[ 0, 1 \right] \times \R^{n}$.
	By taking into account the assumption that $\varphi \in L_{m}^{1} \left( \R^{n} \right)$, the dominated convergence theorem implies
	\begin{align*}
		\lim_{t \to + \infty} 2^{- \left( \abs{\alpha} +m \right)} \sum_{\abs{\beta} =m} \frac{m}{\beta !} \int_{\R^{n}} \int_{0}^{1} \zeta_{q, \beta} \left( \theta, y; t \right) \lvert y^{\beta} \varphi \left( y \right) \rvert {\,} d \theta dy=0.
	\end{align*}
	As a consequence, we obtain
	\begin{align*}
		\lim_{t \to + \infty} t^{\frac{n}{2} \left( 1- \frac{1}{q} \right) + \frac{\abs{\alpha} +m}{2}} \norm{\partial^{\alpha} e^{t \nu \Delta} \varphi - \Lambda_{\alpha, m} \left( t; \varphi \right)}_{q} =0.
	\end{align*}
\end{proof}

\subsection{Commutator estimates with respect to monomial weights} \label{sec:CGL_weight}

As we state in the introduction, we need commutator estimates between the CGL semigroup and monomial weights in $\R^{n}$ to overcome the difficulty arising from the complex scalar field and the dispersion in \eqref{P} for the proof of Theorem \ref{th:P_weight}.
To derive the estimates, we give explicit formulae of the commutation relations.

\begin{thm} \label{th:CGL_commutator}
	Let $m \in \Z_{>0}$, $\varphi \in L_{m}^{1} \left( \R^{n} \right)$, and $\alpha \in \Z_{\geq 0}^{n}$ with $\abs{\alpha} =m$.
	Then, $x^{\alpha} e^{t \nu \Delta} \varphi \in L^{1} \left( \R^{n} \right)$ and the identity
	\begin{align}
		\label{eq:CGL_commutator}
		\left[ x^{\alpha}, e^{t \nu \Delta} \right] \varphi =R_{\alpha} \left( t \right) \varphi
	\end{align}
	holds in $L^{1} \left( \R^{n} \right)$ for any $t>0$, where $\left[ x^{\alpha}, e^{t \nu \Delta} \right] \varphi \coloneqq x^{\alpha} e^{t \nu \Delta} \varphi -e^{t \nu \Delta} x^{\alpha} \varphi$ and
	\begin{align*}
		R_{\alpha} \left( t \right) \varphi \coloneqq \sum_{\substack{\beta + \gamma = \alpha \\ \beta \neq 0}} \frac{\alpha !}{\beta ! \gamma !} \left( -2t \nu \partial \right)^{\beta} e^{t \nu \Delta} x^{\gamma} \varphi + \sum_{\substack{\beta + \gamma \leq \alpha, {\ } \abs{\beta + \gamma} \leq \abs{\alpha} -2 \\ \abs{\beta} +1 \leq \ell \leq \frac{\abs{\alpha} + \abs{\beta} - \abs{\gamma}}{2}}} C_{\ell \beta \gamma}^{\alpha} \left( t \nu \right)^{\ell} \partial^{\beta} e^{t \nu \Delta} x^{\gamma} \varphi
	\end{align*}
	for some $C_{\ell \beta \gamma}^{\alpha} \in \R$ independent of $\nu$, $t$, $x$, and $\varphi$.
\end{thm}

We remark that by virtue of the smoothing effect of the CGL semigroup, each term in $R_{\alpha} \left( t \right) \varphi$ belongs to $L^{1} \left( \R^{n} \right)$ although we do not impose the regularity on $\varphi$.
From this point of view, it is important that differential operators, the CGL semigroup, and monomial weights are in this order.
We also see that $R_{\alpha} \left( t \right) \varphi \in L^{1} \left( \R^{n} \right)$ for any $\alpha \in \Z_{\geq 0}^{n}$ with $\abs{\alpha} =m$ even if $\varphi \in L^{1}_{m-1} \left( \R^{n} \right)$ since monomial weights of order $m$, namely, $x^{\gamma}$ with $\abs{\gamma} =m$, do not appear in $R_{\alpha} \left( t \right) \varphi$.
This fact, which is emphasized quantitatively in the commutator estimates given by Theorem \ref{th:CGL_commutator_esti} below, plays an important role in the proof of Theorem \ref{th:P_weight}.
We cannot obtain such commutation relations for monotonic and positive weights like $\abs{x}^{m}$.
Therefore, monomial weights seem to be more proper to the CGL semigroup than monotonic and positive weights.
We follow the argument in \cite{Kusaba-Ozawa}.

Here and hereafter, $\left( e_{j}; j=1, \ldots, n \right)$ denotes the standard basis of $\R^{n}$.
That is, for each $j \in \left\{ 1, \ldots, n \right\}$, $e_{j} \in \R^{n}$ is a unit vector whose components are $0$ except for the $j$-th coordinate.

\begin{proof}[Proof of Theorem \ref{th:CGL_commutator}]
	We show the assertion by induction on $m \in \Z_{>0}$.
	First, we consider the case where $m=1$.
	Let $\varphi \in L_{1}^{1} \left( \R^{n} \right)$, $t>0$, and $\alpha \in \Z_{\geq 0}^{n}$ with $\abs{\alpha} =1$.
	Then, there exists $j \in \left\{ 1, \ldots, n \right\}$ such that $\alpha =e_{j}$.
	Since $\varphi$, $x_{j} \varphi \in L^{1} \left( \R^{n} \right)$, Lemma \ref{lem:CGL_Lp-Lq} implies $\partial_{j} e^{t \nu \Delta} \varphi$, $e^{t \nu \Delta} x_{j} \varphi \in L^{1} \left( \R^{n} \right)$.
	Moreover, by using the integral representation of the CGL semigroup, we have
	\begin{align*}
		\left( e^{t \nu \Delta} x_{j} \varphi -2t \nu \partial_{j} e^{t \nu \Delta} \varphi \right) \left( x \right)
		&= \int_{\R^{n}} G_{t \nu} \left( x-y \right) y_{j} \varphi \left( y \right) dy-2t \nu \int_{\R^{n}} \left( \partial_{j} G_{t \nu} \right) \left( x-y \right) \varphi \left( y \right) dy \\
		&= \int_{\R^{n}} G_{t \nu} \left( x-y \right) y_{j} \varphi \left( y \right) dy+ \int_{\R^{n}} \left( x_{j} -y_{j} \right) G_{t \nu} \left( x-y \right) \varphi \left( y \right) dy \\
		&=x_{j} \int_{\R^{n}} G_{t \nu} \left( x-y \right) \varphi \left( y \right) dy \\
		&=x_{j} \left( e^{t \nu \Delta} \varphi \right) \left( x \right).
	\end{align*}
	Hence, we deduce that $x_{j} e^{t \nu \Delta} \varphi \in L^{1} \left( \R^{n} \right)$ and
	\begin{align*}
		\left[ x_{j}, e^{t \nu \Delta} \right] \varphi =-2t \nu \partial_{j} e^{t \nu \Delta} \varphi =R_{e_{j}} \left( t \right) \varphi.
	\end{align*}

	Next, we assume that Theorem \ref{th:CGL_commutator} holds for some $m \in \Z_{>0}$.
	Let $\varphi \in L_{m+1}^{1} \left( \R^{n} \right)$, $t>0$, and $\alpha' \in \Z_{\geq 0}^{n}$ with $\abs{\alpha'} =m+1$.
	Then, there exist $\alpha \in \Z_{\geq 0}^{n}$ with $\abs{\alpha} =m$ and $j \in \left\{ 1, \ldots, n \right\}$ such that $\alpha' = \alpha +e_{j}$.
	Since $x^{\alpha} \varphi$, $x_{j} x^{\alpha} \varphi \in L^{1} \left( \R^{n} \right)$, it follows from the case where $m=1$ that $x_{j} e^{t \nu \Delta} x^{\alpha} \varphi \in L^{1} \left( \R^{n} \right)$ and
	\begin{align}
		\label{eq:2.a}
		x_{j} e^{t \nu \Delta} x^{\alpha} \varphi =e^{t \nu \Delta} x_{j} x^{\alpha} \varphi -2t \nu \partial_{j} e^{t \nu \Delta} x^{\alpha} \varphi =e^{t \nu \Delta} x^{\alpha'} \varphi -2t \nu \partial_{j} e^{t \nu \Delta} x^{\alpha} \varphi.
	\end{align}
	Now, we show $x_{j} R_{\alpha} \left( t \right) \varphi \in L^{1} \left( \R^{n} \right)$.
	To this end, it suffices to prove that $x_{j} \partial^{\beta} e^{t \nu \Delta} x^{\gamma} \varphi \in L^{1} \left( \R^{n} \right)$ for any $\beta, \gamma \in \Z_{\geq 0}^{n}$ with $\beta + \gamma \leq \alpha$ and $\abs{\gamma} \leq m-1$.
	Let $\beta, \gamma \in \Z_{\geq 0}^{n}$ satisfy $\beta + \gamma \leq \alpha$ and $\abs{\gamma} \leq m-1$.
	Since $x^{\gamma} \varphi$, $x_{j} x^{\gamma} \varphi \in L^{1} \left( \R^{n} \right)$, Theorem \ref{th:CGL_commutator} with $m=1$ implies $x_{j} e^{t \nu \Delta} x^{\gamma} \varphi \in L^{1} \left( \R^{n} \right)$ and
	\begin{align}
		\label{eq:2.b}
		x_{j} e^{t \nu \Delta} x^{\gamma} \varphi =e^{t \nu \Delta} x_{j} x^{\gamma} \varphi -2t \nu \partial_{j} e^{t \nu \Delta} x^{\gamma} \varphi =e^{t \nu \Delta} x^{\gamma +e_{j}} \varphi -2t \nu \partial_{j} e^{t \nu \Delta} x^{\gamma} \varphi.
	\end{align}
	In addition, it follows from Lemma \ref{lem:CGL_Lp-Lq} that each term on the right hand side of the above identity belongs to $W^{m, 1} \left( \R^{n} \right)$.
	Therefore, we have $\partial^{\beta} \left( x_{j} e^{t \nu \Delta} x^{\gamma} \varphi \right) \in L^{1} \left( \R^{n} \right)$ and
	\begin{align*}
		\partial^{\beta} \left( x_{j} e^{t \nu \Delta} x^{\gamma} \varphi \right)
		&= \partial^{\beta} \left( e^{t \nu \Delta} x^{\gamma +e_{j}} \varphi -2t \nu \partial_{j} e^{t \nu \Delta} x^{\gamma} \varphi \right) \\
		&= \partial^{\beta} e^{t \nu \Delta} x^{\gamma +e_{j}} \varphi -2t \nu \partial^{\beta +e_{j}} e^{t \nu \Delta} x^{\gamma} \varphi.
	\end{align*}
	On the other hand, by a simple calculation, we obtain
	\begin{align}
		\label{eq:2.c}
		\partial^{\beta} \left( x_{j} e^{t \nu \Delta} x^{\gamma} \varphi \right) = \begin{dcases}
			x_{j} \partial^{\beta} e^{t \nu \Delta} x^{\gamma} \varphi + \beta_{j} \partial^{\beta -e_{j}} e^{t \nu \Delta} x^{\gamma} \varphi &\quad \text{if} \quad \beta_{j} \geq 1, \\
			x_{j} \partial^{\beta} e^{t \nu \Delta} x^{\gamma} \varphi &\quad \text{if} \quad \beta_{j} =0,
		\end{dcases}
	\end{align}
	where $\beta_{j}$ is the $j$-th component of $\beta$.
	Since $\partial^{\beta -e_{j}} e^{t \nu \Delta} x^{\gamma} \varphi \in L^{1} \left( \R^{n} \right)$ if $e_{j} \leq \beta$, we deduce that $x_{j} \partial^{\beta} e^{t \nu \Delta} x^{\gamma} \varphi \in L^{1} \left( \R^{n} \right)$ in any case, which in turn implies $x_{j} R_{\alpha} \left( t \right) \varphi \in L^{1} \left( \R^{n} \right)$.
	Taking into account the fact that $x_{j} e^{t \nu \Delta} x^{\alpha} \varphi$, $x_{j} R_{\alpha} \left( t \right) \varphi \in L^{1} \left( \R^{n} \right)$, by the induction hypothesis and \eqref{eq:2.a}, we have $x^{\alpha'} e^{t \nu \Delta} \varphi \in L^{1} \left( \R^{n} \right)$ and
	\begin{align*}
		x^{\alpha'} e^{t \nu \Delta} \varphi
		&=x_{j} x^{\alpha} e^{t \nu \Delta} \varphi \\
		&=x_{j} e^{t \nu \Delta} x^{\alpha} \varphi +x_{j} R_{\alpha} \left( t \right) \varphi \\
		&=e^{t \nu \Delta} x^{\alpha'} \varphi -2t \nu \partial_{j} e^{t \nu \Delta} x^{\alpha} \varphi +x_{j} R_{\alpha} \left( t \right) \varphi,
	\end{align*}
	whence follows
	\begin{align*}
		\bigl[ x^{\alpha'}, e^{t \nu \Delta} \bigr] {\,} \varphi =-2t \nu \partial_{j} e^{t \nu \Delta} x^{\alpha} \varphi +x_{j} R_{\alpha} \left( t \right) \varphi.
	\end{align*}
	In order to complete the proof, we show that the right hand side on the above identity is represented as the form
	\begin{align*}
		R_{\alpha'} \left( t \right) \varphi = \sum_{\substack{\beta' + \gamma' = \alpha' \\ \beta' \neq 0}} \frac{\alpha' !}{\beta' ! \gamma' !} \left( -2t \nu \partial \right)^{\beta'} e^{t \nu \Delta} x^{\gamma'} \varphi + \sum_{\substack{\beta' + \gamma' \leq \alpha', {\ } \abs{\beta' + \gamma'} \leq \abs{\alpha'} -2 \\ \abs{\beta'} +1 \leq \ell' \leq \frac{\lvert \alpha' \rvert + \lvert \beta' \rvert - \lvert \gamma' \rvert}{2}}} C_{\ell' \beta' \gamma'}^{\alpha'} \left( t \nu \right)^{\ell'} \partial^{\beta'} e^{t \nu \Delta} x^{\gamma'} \varphi.
	\end{align*}
	Firstly, it is clear that the term $-2t \nu \partial_{j} e^{t \nu \Delta} x^{\alpha} \varphi$ is a part of the first sum in $R_{\alpha'} \left( t \right) \varphi$ with $\left( \beta', \gamma' \right) = \left( e_{j}, \alpha \right)$.
	Secondly, we take $\beta, \gamma \in \Z_{\geq 0}^{n}$ satisfying $\beta + \gamma = \alpha$ and $\beta \neq 0$.
	If $\beta_{j} \geq 1$, then it follows from \eqref{eq:2.b} and \eqref{eq:2.c} that
	\begin{align}
		\label{eq:2.d}
		&x_{j} \left( -2t \nu \partial \right)^{\beta} e^{t \nu \Delta} x^{\gamma} \varphi \nonumber \\
		&\hspace{1cm} = \left( -2t \nu \partial \right)^{\beta} \left( x_{j} e^{t \nu \Delta} x^{\gamma} \varphi \right) - \beta_{j} \left( -2t \nu \right)^{\abs{\beta}} \partial^{\beta -e_{j}} e^{t \nu \Delta} x^{\gamma} \varphi \nonumber \\
		&\hspace{1cm} = \left( -2t \nu \partial \right)^{\beta} \left( e^{t \nu \Delta} x^{\gamma +e_{j}} \varphi -2t \nu \partial_{j} e^{t \nu \Delta} x^{\gamma} \varphi \right) - \beta_{j} \left( -2 \right)^{\abs{\beta}} \left( t \nu \right)^{\abs{\beta}} \partial^{\beta -e_{j}} e^{t \nu \Delta} x^{\gamma} \varphi \nonumber \\
		&\hspace{1cm} = \left( -2t \nu \partial \right)^{\beta} e^{t \nu \Delta} x^{\gamma +e_{j}} \varphi + \left( -2t \nu \partial \right)^{\beta +e_{j}} e^{t \nu \Delta} x^{\gamma} \varphi - \beta_{j} \left( -2 \right)^{\abs{\beta}} \left( t \nu \right)^{\abs{\beta}} \partial^{\beta -e_{j}} e^{t \nu \Delta} x^{\gamma} \varphi.
	\end{align}
	Similarly, if $\beta_{j} =0$, then we have
	\begin{align}
		\label{eq:2.e}
		x_{j} \left( -2t \nu \partial \right)^{\beta} e^{t \nu \Delta} x^{\gamma} \varphi
		&= \left( -2t \nu \partial \right)^{\beta} \left( x_{j} e^{t \nu \Delta} x^{\gamma} \varphi \right) \nonumber \\
		&= \left( -2t \nu \partial \right)^{\beta} \left( e^{t \nu \Delta} x^{\gamma +e_{j}} \varphi -2t \nu \partial_{j} e^{t \nu \Delta} x^{\gamma} \varphi \right) \nonumber \\
		&= \left( -2t \nu \partial \right)^{\beta} e^{t \nu \Delta} x^{\gamma +e_{j}} \varphi + \left( -2t \nu \partial \right)^{\beta +e_{j}} e^{t \nu \Delta} x^{\gamma} \varphi.
	\end{align}
	The terms $\left( -2t \nu \partial \right)^{\beta} e^{t \nu \Delta} x^{\gamma +e_{j}} \varphi$ and $\left( -2t \nu \partial \right)^{\beta +e_{j}} e^{t \nu \Delta} x^{\gamma} \varphi$ in \eqref{eq:2.d} or \eqref{eq:2.e} are parts of the first sum in $R_{\alpha'} \left( t \right) \varphi$ with $\left( \beta', \gamma' \right) = \left( \beta, \gamma +e_{j} \right)$, $\left( \beta +e_{j}, \gamma \right)$, respectively.
	Moreover, we have
	\begin{align*}
		&-2t \nu \partial_{j} e^{t \nu \Delta} x^{\alpha} \varphi + \sum_{\substack{\beta + \gamma = \alpha \\ \beta \neq 0}} \frac{\alpha !}{\beta ! \gamma !} \left( \left( -2t \nu \partial \right)^{\beta} e^{t \nu \Delta} x^{\gamma +e_{j}} \varphi + \left( -2t \nu \partial \right)^{\beta +e_{j}} e^{t \nu \Delta} x^{\gamma} \varphi \right) \\
		&\hspace{1cm} = \sum_{\substack{\beta + \gamma = \alpha \\ \beta \neq 0}} \frac{\alpha !}{\beta ! \gamma !} \left( -2t \nu \partial \right)^{\beta} e^{t \nu \Delta} x^{\gamma +e_{j}} \varphi + \sum_{\beta + \gamma = \alpha} \frac{\alpha !}{\beta ! \gamma !} \left( -2t \nu \partial \right)^{\beta +e_{j}} e^{t \nu \Delta} x^{\gamma} \varphi \\
		&\hspace{1cm} = \sum_{\substack{\beta' + \gamma' = \alpha' \\ 0 \neq \beta' \leq \alpha}} \frac{\alpha !}{\beta' ! \left( \gamma' -e_{j} \right) !} \left( -2t \nu \partial \right)^{\beta'} e^{t \nu \Delta} x^{\gamma'} \varphi + \sum_{\substack{\beta' + \gamma' = \alpha' \\ \gamma' \leq \alpha}} \frac{\alpha !}{\left( \beta' -e_{j} \right) ! \gamma' !} \left( -2t \nu \partial \right)^{\beta'} e^{t \nu \Delta} x^{\gamma'} \varphi \\
		&\hspace{1cm} = \sum_{\substack{\beta' + \gamma' = \alpha' \\ \beta' \neq 0}} \left( \binom{\alpha}{\beta'} + \binom{\alpha}{\gamma'} \right) \left( -2t \nu \partial \right)^{\beta'} e^{t \nu \Delta} x^{\gamma'} \varphi \\
		&\hspace{1cm} = \sum_{\substack{\beta' + \gamma' = \alpha' \\ \beta' \neq 0}} \frac{\alpha' !}{\beta' ! \gamma' !} \left( -2t \nu \partial \right)^{\beta'} e^{t \nu \Delta} x^{\gamma'} \varphi,
	\end{align*}
	which implies that all components of the first sum in $R_{\alpha'} \left( t \right) \varphi$ have appeared.
	On the other hand, taking $\left( \ell', \beta', \gamma' \right) = \left( \abs{\beta}, \beta -e_{j}, \gamma \right)$ with $e_{j} \leq \beta$, we see that
	\begin{itemize}
		\item
			$\beta ' + \gamma' = \beta -e_{j} + \gamma = \alpha -e_{j} = \alpha' -2e_{j} \leq \alpha'$,
		\item
			$\abs{\beta' + \gamma'} = \abs{\beta -e_{j} + \gamma} = \abs{\alpha} -1= \abs{\alpha'} -2$,
		\item
			$\abs{\beta'} +1= \abs{\beta -e_{j}} +1= \abs{\beta} = \ell'$,
		\item
			$\ell' \leq \dfrac{\abs{\alpha'} + \abs{\beta'} - \abs{\gamma'}}{2} {\ } \Leftrightarrow {\ } \abs{\beta} \leq \dfrac{\left( \abs{\alpha} +1 \right) + \left( \abs{\beta} -1 \right) - \abs{\gamma}}{2} = \dfrac{\abs{\alpha} + \abs{\beta} - \abs{\gamma}}{2} = \abs{\beta}$.
	\end{itemize}
	This means that the term $- \beta_{j} \left( -2 \right)^{\abs{\beta}} \left( t \nu \right)^{\abs{\beta}} \partial^{\beta -e_{j}} e^{t \nu \Delta} x^{\gamma} \varphi$ in \eqref{eq:2.d} is a part of the second sum in $R_{\alpha'} \left( t \right) \varphi$ with $\left( \ell', \beta', \gamma' \right) = \left( \abs{\beta}, \beta -e_{j}, \gamma \right)$.
	Therefore, we conclude that all components of the first sum in $x_{j} R_{\alpha} \left( t \right) \varphi$ are represented as parts of $R_{\alpha'} \left( t \right) \varphi$.
	Finally, we take $\beta, \gamma \in \Z_{\geq 0}^{n}$ and $\ell \in \Z_{\geq 0}$ satisfying
	\begin{align*}
		\beta + \gamma \leq \alpha, \qquad \abs{\beta + \gamma} \leq \abs{\alpha} -2, \qquad \abs{\beta} +1 \leq \ell \leq \frac{\abs{\alpha} + \abs{\beta} - \abs{\gamma}}{2}.
	\end{align*}
	If $\beta_{j} \geq 1$, then it follows from \eqref{eq:2.b} and \eqref{eq:2.c} that
	\begin{align}
		\label{eq:2.f}
		x_{j} \left( t \nu \right)^{\ell} \partial^{\beta} e^{t \nu \Delta} x^{\gamma} \varphi
		&= \left( t \nu \right)^{\ell} \partial^{\beta} \left( x_{j} e^{t \nu \Delta} x^{\gamma} \varphi \right) - \beta_{j} \left( t \nu \right)^{\ell} \partial^{\beta -e_{j}} e^{t \nu \Delta} x^{\gamma} \varphi \nonumber \\
		&= \left( t \nu \right)^{\ell} \partial^{\beta} \left( e^{t \nu \Delta} x^{\gamma +e_{j}} \varphi -2t \nu \partial_{j} e^{t \nu \Delta} x^{\gamma} \varphi \right) - \beta_{j} \left( t \nu \right)^{\ell} \partial^{\beta -e_{j}} e^{t \nu \Delta} x^{\gamma} \varphi \nonumber \\
		&= \left( t \nu \right)^{\ell} \partial^{\beta} e^{t \nu \Delta} x^{\gamma +e_{j}} \varphi -2 \left( t \nu \right)^{\ell +1} \partial^{\beta +e_{j}} e^{t \nu \Delta} x^{\gamma} \varphi - \beta_{j} \left( t \nu \right)^{\ell} \partial^{\beta -e_{j}} e^{t \nu \Delta} x^{\gamma} \varphi.
	\end{align}
	In the same way, if $\beta_{j} =0$, then we have
	\begin{align}
		\label{eq:2.g}
		x_{j} \left( t \nu \right)^{\ell} \partial^{\beta} e^{t \nu \Delta} x^{\gamma} \varphi
		&= \left( t \nu \right)^{\ell} \partial^{\beta} \left( x_{j} e^{t \nu \Delta} x^{\gamma} \varphi \right) \nonumber \\
		&= \left( t \nu \right)^{\ell} \partial^{\beta} \left( e^{t \nu \Delta} x^{\gamma +e_{j}} \varphi -2t \nu \partial_{j} e^{t \nu \Delta} x^{\gamma} \varphi \right) \nonumber \\
		&= \left( t \nu \right)^{\ell} \partial^{\beta} e^{t \nu \Delta} x^{\gamma +e_{j}} \varphi -2 \left( t \nu \right)^{\ell +1} \partial^{\beta +e_{j}} e^{t \nu \Delta} x^{\gamma} \varphi.
	\end{align}
	Taking $\left( \ell', \beta', \gamma' \right) = \left( \ell, \beta, \gamma +e_{j} \right)$, we see that
	\begin{itemize}
		\item
			$\beta ' + \gamma' = \beta + \gamma +e_{j} \leq \alpha +e_{j} = \alpha'$,
		\item
			$\abs{\beta' + \gamma'} = \abs{\beta + \gamma +e_{j}} = \abs{\beta + \gamma} +1 \leq \abs{\alpha} -1= \abs{\alpha'} -2$,
		\item
			$\abs{\beta'} +1= \abs{\beta} +1 \leq \ell = \ell'$,
		\item
			$\ell' \leq \dfrac{\abs{\alpha'} + \abs{\beta'} - \abs{\gamma'}}{2} {\ } \Leftrightarrow {\ } \ell \leq \dfrac{\left( \abs{\alpha} +1 \right) + \abs{\beta} - \left( \abs{\gamma} +1 \right)}{2} = \dfrac{\abs{\alpha} + \abs{\beta} - \abs{\gamma}}{2}$.
	\end{itemize}
	This means that the term $\left( t \nu \right)^{\ell} \partial^{\beta} e^{t \nu \Delta} x^{\gamma +e_{j}} \varphi$ in \eqref{eq:2.f} or \eqref{eq:2.g} is a part of the second sum in $R_{\alpha'} \left( t \right) \varphi$ with $\left( \ell', \beta', \gamma' \right) = \left( \ell, \beta, \gamma +e_{j} \right)$.
	Next, taking $\left( \ell', \beta', \gamma' \right) = \left( \ell +1, \beta +e_{j}, \gamma \right)$, we have
	\begin{itemize}
		\item
			$\beta ' + \gamma' = \beta +e_{j} + \gamma \leq \alpha +e_{j} = \alpha'$,
		\item
			$\abs{\beta' + \gamma'} = \abs{\beta +e_{j} + \gamma} = \abs{\beta + \gamma} +1 \leq \abs{\alpha} -1= \abs{\alpha'} -2$,
		\item
			$\abs{\beta'} +1= \abs{\beta +e_{j}} +1= \left( \abs{\beta} +1 \right) +1 \leq \ell +1= \ell'$,
		\item
			$\ell' \leq \dfrac{\abs{\alpha'} + \abs{\beta'} - \abs{\gamma'}}{2} {\ } \Leftrightarrow {\ } \ell \leq \dfrac{\left( \abs{\alpha} +1 \right) + \left( \abs{\beta} +1 \right) - \abs{\gamma}}{2} -1= \dfrac{\abs{\alpha} + \abs{\beta} - \abs{\gamma}}{2}$.
	\end{itemize}
	Hence, the term $-2 \left( t \nu \right)^{\ell +1} \partial^{\beta +e_{j}} e^{t \nu \Delta} x^{\gamma} \varphi$ in \eqref{eq:2.f} or \eqref{eq:2.g} is a part of the second sum in $R_{\alpha'} \left( t \right) \varphi$ with $\left( \ell', \beta', \gamma' \right) = \left( \ell +1, \beta +e_{j}, \gamma \right)$.
	Finally, taking $\left( \ell', \beta', \gamma' \right) = \left( \ell, \beta -e_{j}, \gamma \right)$ with $e_{j} \leq \beta$, we obtain
	\begin{itemize}
		\item
			$\beta ' + \gamma' = \beta -e_{j} + \gamma \leq \alpha -e_{j} = \alpha' -2e_{j} \leq \alpha'$,
		\item
			$\abs{\beta' + \gamma'} = \abs{\beta -e_{j} + \gamma} = \abs{\beta + \gamma} -1 \leq \abs{\alpha} -3= \abs{\alpha'} -4 \leq \abs{\alpha'} -2$,
		\item
			$\abs{\beta'} +1= \abs{\beta -e_{j}} +1 = \abs{\beta} \leq \ell -1 \leq \ell = \ell'$,
		\item
			$\ell' \leq \dfrac{\abs{\alpha'} + \abs{\beta'} - \abs{\gamma'}}{2} {\ } \Leftrightarrow {\ } \ell \leq \dfrac{\left( \abs{\alpha} +1 \right) + \left( \abs{\beta} -1 \right) - \abs{\gamma}}{2} = \dfrac{\abs{\alpha} + \abs{\beta} - \abs{\gamma}}{2}$.
	\end{itemize}
	This implies that the term $- \beta_{j} \left( t \nu \right)^{\ell} \partial^{\beta -e_{j}} e^{t \nu \Delta} x^{\gamma} \varphi$ in \eqref{eq:2.f} is a part of the second sum in $R_{\alpha'} \left( t \right) \varphi$ with $\left( \ell', \beta', \gamma' \right) = \left( \ell, \beta -e_{j}, \gamma \right)$.
	As a consequence, all components of the second sum in $x_{j} R_{\alpha} \left( t \right) \varphi$ are represented as parts of the second sum in $R_{\alpha'} \left( t \right) \varphi$.
	This completes the proof.
\end{proof}

\begin{rem}
	In the case where $\alpha =e_{j}$ with $j \in \left\{ 1, \ldots, n \right\}$, \eqref{eq:CGL_commutator} is represented as
	\begin{align*}
		\left[ x_{j}, e^{t \nu \Delta} \right] \varphi =-2t \nu \partial_{j} e^{t \nu \Delta} \varphi.
	\end{align*}
	Under suitable assumptions, the above identity is valid even for the case where $\re \nu =0$ and in particular for the Schr\"{o}dinger evolution group $\left( e^{it \Delta}; t \in \R \right)$.
	Such commutation relations are useful for studying smoothing effects of Schr\"{o}dinger evolution groups associated with Hamiltonians involving a potential \cite{Jensen}.
	See also \cite{Ozawa}.
\end{rem}

By using Theorem \ref{th:CGL_commutator}, we can derive the commutator estimates of the CGL semigroup and monomial weights.

\begin{thm} \label{th:CGL_commutator_esti}
	Let $m \in \Z_{>0}$.
	Then, there exists $C>0$ such that the estimate
	\begin{align}
		\label{eq:CGL_commutator_esti}
		\sum_{\abs{\alpha} =m} \norm{\left[ x^{\alpha}, e^{t \nu \Delta} \right] \varphi}_{1} \leq C \left\{ t^{\frac{1}{2}} \norm{\lvert x \rvert^{m-1} {\,} \varphi}_{1} + \left( t^{\frac{1}{2}} +t^{\frac{m}{2}} \right) \norm{\varphi}_{1} \right\}
	\end{align}
	holds for any $\varphi \in L_{m}^{1} \left( \R^{n} \right)$ and $t>0$.
\end{thm}

By a simple calculation, we see that there exists $C>0$ such that the estimate
\begin{align}
	\label{eq:CGL_weight}
	\sum_{\abs{\alpha} =m} \norm{x^{\alpha} e^{t \nu \Delta} \varphi}_{1} \leq C \left( \norm{\abs{x}^{m} \varphi}_{1} +t^{\frac{m}{2}} \norm{\varphi}_{1} \right)
\end{align}
holds for any $\varphi \in L_{m}^{1} \left( \R^{n} \right)$ and $t>0$.
Comparing \eqref{eq:CGL_commutator_esti} with \eqref{eq:CGL_weight}, Theorem \ref{th:CGL_commutator_esti} enables us to control the CGL semigroup with the weights of order $m$ by not $\norm{\abs{x}^{m} \varphi}_{1}$ but $\lVert \lvert x \rvert^{m-1} {\,} \varphi \rVert_{1}$.
From this difference, we can understand quantitatively an advantage of monomial weights over monotonic and positive weights to the CGL semigroup.

\begin{proof}[Proof of Theorem \ref{th:CGL_commutator_esti}]
	Let $\varphi \in L_{m}^{1} \left( \R^{n} \right)$ and let $t>0$.
	By Theorem \ref{th:CGL_commutator}, the identity
	\begin{align*}
		\left[ x^{\alpha}, e^{t \nu \Delta} \right] \varphi =R_{\alpha} \left( t \right) \varphi
	\end{align*}
	holds for any $\alpha \in \Z_{\geq 0}^{n}$ with $\abs{\alpha} =m$.
	For the case where $m=1$, Lemma \ref{lem:CGL_Lp-Lq} implies
	\begin{align*}
		\sum_{\abs{\alpha} =1} \norm{\left[ x^{\alpha}, e^{t \nu \Delta} \right] \varphi}_{1}
		&= \sum_{j=1}^{n} \norm{R_{e_{j}} \left( t \right) \varphi}_{1} \\
		&=2 \abs{\nu} t \sum_{j=1}^{n} \norm{\partial_{j} e^{t \nu \Delta} \varphi}_{1} \\
		&\leq \left( 2 \abs{\nu} \sum_{j=1}^{n} \norm{\partial_{j} G_{\nu}}_{1} \right) t^{\frac{1}{2}} \norm{\varphi}_{1}.
	\end{align*}
	Next, we consider the case where $m \geq 2$.
	Here and hereafter, let $C$ denote a positive constant independent of $\varphi$ and $t$ which may change from line to line.
	It follows from Lemma \ref{lem:CGL_Lp-Lq} that
	\begin{align*}
		\norm{R_{\alpha} \left( t \right) \varphi}_{1} &\leq C \sum_{\substack{\beta + \gamma = \alpha \\ \beta \neq 0}} t^{\abs{\beta}} {\,} \lVert \partial^{\beta} e^{t \nu \Delta} x^{\gamma} \varphi \rVert_{1} +C \sum_{\substack{\beta + \gamma \leq \alpha, {\ } \abs{\beta + \gamma} \leq m-2 \\ \abs{\beta} +1 \leq \ell \leq \frac{m+ \abs{\beta} - \abs{\gamma}}{2}}} t^{\ell} {\,} \lVert \partial^{\beta} e^{t \nu \Delta} x^{\gamma} \varphi \rVert_{1} \\
		&\leq C \sum_{\substack{\abs{\beta} + \abs{\gamma} =m \\ \abs{\beta} \geq 1}} t^{\frac{\abs{\beta}}{2}} \norm{x^{\gamma} \varphi}_{1} +C \sum_{\substack{\abs{\beta} + \abs{\gamma} \leq m-2 \\ \abs{\beta} +1 \leq \ell \leq \frac{m+ \abs{\beta} - \abs{\gamma}}{2}}} t^{\ell - \frac{\abs{\beta}}{2}} \norm{x^{\gamma} \varphi}_{1}
	\end{align*}
	for any $\alpha \in \Z_{\geq 0}^{n}$ with $\abs{\alpha} =m$.
	Therefore, we have
	\begin{align}
		\label{eq:2.h}
		\sum_{\abs{\alpha} =m} \norm{\left[ x^{\alpha}, e^{t \nu \Delta} \right] \varphi}_{1} \leq C \sum_{\substack{\abs{\beta} + \abs{\gamma} =m \\ \abs{\beta} \geq 1}} t^{\frac{\abs{\beta}}{2}} \norm{x^{\gamma} \varphi}_{1} +C \sum_{j=0}^{m-2} \sum_{\abs{\beta} + \abs{\gamma} =j} \sum_{\abs{\beta} +1 \leq \ell \leq \frac{m+ \abs{\beta} - \abs{\gamma}}{2}} t^{\ell - \frac{\abs{\beta}}{2}} \norm{x^{\gamma} \varphi}_{1}.
	\end{align}
	Now, we take $\beta, \gamma \in \Z_{\geq 0}^{n}$ satisfying $\abs{\beta} + \abs{\gamma} =m$ and $\abs{\beta} \geq 1$.
	Then, we see that
	\begin{align*}
		\frac{\abs{\beta} -1}{m-1} + \frac{\abs{\gamma}}{m-1} =1, \qquad 0 \leq \frac{\abs{\beta} -1}{m-1} \leq 1, \qquad 0 \leq \frac{\abs{\gamma}}{m-1} \leq 1.
	\end{align*}
	Hence, by H\"{o}lder's inequality, we have
	\begin{align}
		\label{eq:2.i}
		t^{\frac{\abs{\beta}}{2}} \norm{x^{\gamma} \varphi}_{1} &\leq t^{\frac{\abs{\beta}}{2}} {\,} \bigl\lVert \abs{x}^{\abs{\gamma}} \varphi \bigr\rVert_{1} \nonumber \\
		&\leq t^{\frac{\abs{\beta}}{2}} \norm{\varphi}_{1}^{\frac{\abs{\beta} -1}{m-1}} \norm{\lvert x \rvert^{m-1} {\,} \varphi}_{1}^{\frac{\abs{\gamma}}{m-1}} \nonumber \\
		&=t^{\frac{1}{2}} \left( t^{\frac{m-1}{2}} \norm{\varphi}_{1} \right)^{\frac{\abs{\beta} -1}{m-1}} \norm{\lvert x \rvert^{m-1} {\,} \varphi}_{1}^{\frac{\abs{\gamma}}{m-1}} \nonumber \\
		&\leq t^{\frac{1}{2}} \left( t^{\frac{m-1}{2}} \norm{\varphi}_{1} + \norm{\lvert x \rvert^{m-1} {\,} \varphi}_{1} \right) \nonumber \\
		&=t^{\frac{m}{2}} \norm{\varphi}_{1} +t^{\frac{1}{2}} \norm{\lvert x \rvert^{m-1} {\,} \varphi}_{1}.
	\end{align}
	Next, we take $j, \ell \in \Z_{\geq 0}$ and $\beta, \gamma \in \Z_{\geq 0}^{n}$ satisfying
	\begin{align*}
		0 \leq j \leq m-2, \qquad \abs{\beta} + \abs{\gamma} =j, \qquad \abs{\beta} +1 \leq \ell \leq \frac{m+ \abs{\beta} - \abs{\gamma}}{2}.
	\end{align*}
	Then, we see that
	\begin{gather*}
		\frac{m+ \abs{\beta} -j-1}{m-1} + \frac{\abs{\gamma}}{m-1} =1, \qquad 0 \leq \frac{m+ \abs{\beta} -j-1}{m-1} \leq 1, \qquad 0 \leq \frac{\abs{\gamma}}{m-1} \leq 1, \\
		\abs{\beta} +1 \leq 2 \ell - \abs{\beta} -1 \leq m+ \abs{\beta} - \left( \abs{\gamma} + \abs{\beta} \right) -1=m+ \abs{\beta} -j-1.
	\end{gather*}
	In particular, we have
	\begin{align*}
		0 \leq \frac{\abs{\beta} +1}{m+ \abs{\beta} -j-1} \leq \frac{2 \ell - \abs{\beta} -1}{m+ \abs{\beta} -j-1} \leq 1.
	\end{align*}
	Thus, it follows from H\"{o}lder's inequality that
	\begin{align}
		\label{eq:2.j}
		t^{\ell - \frac{\abs{\beta}}{2}} \norm{x^{\gamma} \varphi}_{1} &\leq t^{\ell - \frac{\abs{\beta}}{2}} {\,} \bigl\lVert \abs{x}^{\abs{\gamma}} \varphi \bigr\rVert_{1} \nonumber \\
		&\leq t^{\frac{2 \ell - \abs{\beta}}{2}} \norm{\varphi}_{1}^{\frac{m+ \abs{\beta} -j-1}{m-1}} \norm{\lvert x \rvert^{m-1} {\,} \varphi}_{1}^{\frac{\abs{\gamma}}{m-1}} \nonumber \\
		&=t^{\frac{1}{2}} \left( t^{\frac{\left( m-1 \right) \left( 2 \ell - \abs{\beta} -1 \right)}{2 \left( m+ \abs{\beta} -j-1 \right)}} \norm{\varphi}_{1} \right)^{\frac{m+ \abs{\beta} -j-1}{m-1}} \norm{\lvert x \rvert^{m-1} {\,} \varphi}_{1}^{\frac{\abs{\gamma}}{m-1}} \nonumber \\
		&\leq t^{\frac{1}{2}} \left( t^{\frac{\left( m-1 \right) \left( 2 \ell - \abs{\beta} -1 \right)}{2 \left( m+ \abs{\beta} -j-1 \right)}} \norm{\varphi}_{1} + \norm{\lvert x \rvert^{m-1} {\,} \varphi}_{1} \right) \nonumber \\
		&\leq \left( t^{\frac{1}{2}} + t^{\frac{m}{2}} \right) \norm{\varphi}_{1} +t^{\frac{1}{2}} \norm{\lvert x \rvert^{m-1} {\,} \varphi}_{1}.
	\end{align}
	Here, we have used the inequality
	\begin{align*}
		t^{\frac{\left( m-1 \right) \left( 2 \ell - \abs{\beta} -1 \right)}{2 \left( m+ \abs{\beta} -j-1 \right)}} \leq 1+t^{\frac{m-1}{2}}.
	\end{align*}
	Finally, combining \eqref{eq:2.h}, \eqref{eq:2.i}, and \eqref{eq:2.j} yields the desired estimate.
\end{proof}

\begin{rem} \label{rem:CGL_remainder_esti}
	Under the same assumption as in the Theorem \ref{th:CGL_commutator}, from the proofs of Theorems \ref{th:CGL_commutator} and \ref{th:CGL_commutator_esti}, we see that $x_{j} R_{\alpha} \left( t \right) \varphi$ is represented as a part of $R_{\alpha +e_{j}} \left( t \right) \varphi$ for any $j \in \left\{ 1, \ldots, n \right\}$ and that there exists $C>0$ independent of $\varphi$ such that the estimate
	\begin{align*}
		\sum_{j=1}^{n} \sum_{\abs{\alpha} =m} \norm{x_{j} R_{\alpha} \left( t \right) \varphi}_{1} \leq C \left\{ t^{\frac{1}{2}} \norm{\abs{x}^{m} \varphi}_{1} + \left( t^{\frac{1}{2}} +t^{\frac{m+1}{2}} \right) \norm{\varphi}_{1} \right\}
	\end{align*}
	holds for any $t>0$.
\end{rem}

Next lemma provides commutator estimates between the CGL semigroup and general weights, which will be used in the proof of Theorem \ref{th:P_weight} to compute weighted estimates with approximation.

\begin{lem} \label{lem:CGL_weight_appro}
	Let $w \in W^{2, \infty} \left( \R^{n} \right)$ and let $\varphi \in L^{1} \left( \R^{n} \right)$.
	Then, the estimate
	\begin{align*}
		\norm{\left[ w, e^{t \nu \Delta} \right] \varphi}_{1} &\leq \abs{\nu} \norm{G_{\nu}}_{1} \left( \norm{\Delta w}_{\infty} \norm{G_{\nu}}_{1} t+4 \norm{\nabla w}_{\infty} \norm{\nabla G_{\nu}}_{1} t^{\frac{1}{2}} \right) \norm{\varphi}_{1}
	\end{align*}
	holds for any $t>0$, where $\left[ w, e^{t \nu \Delta} \right] \varphi \coloneqq we^{t \nu \Delta} \varphi -e^{t \nu \Delta} w \varphi$.
\end{lem}

\begin{proof}
	We start with the identity
	\begin{align*}
		\left[ w, e^{t \nu \Delta} \right] \varphi
		&= \int_{0}^{t} \frac{d}{ds} \left( e^{\left( t-s \right) \nu \Delta} we^{s \nu \Delta} \varphi \right) ds \\
		&= \int_{0}^{t} e^{\left( t-s \right) \nu \Delta} \left( - \nu \Delta \left( we^{s \nu \Delta} \varphi \right) +w \nu \Delta e^{s \nu \Delta} \varphi \right) ds \\
		&= \nu \int_{0}^{t} e^{\left( t-s \right) \nu \Delta} \left( - \Delta we^{s \nu \Delta} \varphi -2 \nabla w \cdot \nabla e^{s \nu \Delta} \varphi \right) ds.
	\end{align*}
	By taking $L^{1}$-norm of the above identity and applying Lemma \ref{lem:CGL_Lp-Lq}, we have
	\begin{align*}
		\norm{\left[ w, e^{t \nu \Delta} \right] \varphi}_{1} &\leq \abs{\nu} \int_{0}^{t} \norm{e^{\left( t-s \right) \nu \Delta} \left( - \Delta we^{s \nu \Delta} \varphi -2 \nabla w \cdot \nabla e^{s \nu \Delta} \varphi \right)}_{1} ds \\
		&\leq \abs{\nu} \norm{G_{\nu}}_{1} \int_{0}^{t} \norm{- \Delta we^{s \nu \Delta} \varphi -2 \nabla w \cdot \nabla e^{s \nu \Delta} \varphi}_{1} ds \\
		&\leq \abs{\nu} \norm{G_{\nu}}_{1} \left( \norm{\Delta w}_{\infty} \int_{0}^{t} \norm{e^{s \nu \Delta} \varphi}_{1} ds+2 \norm{\nabla w}_{\infty} \int_{0}^{t} \norm{\nabla e^{s \nu \Delta} \varphi}_{1} ds \right) \\
		&\leq \abs{\nu} \norm{G_{\nu}}_{1} \left( \norm{\Delta w}_{\infty} \norm{G_{\nu}}_{1} \norm{\varphi}_{1} \int_{0}^{t} ds+2 \norm{\nabla w}_{\infty} \norm{\nabla G_{\nu}}_{1} \norm{\varphi}_{1} \int_{0}^{t} s^{- \frac{1}{2}} ds \right) \\
		&= \abs{\nu} \norm{G_{\nu}}_{1} \left( \norm{\Delta w}_{\infty} \norm{G_{\nu}}_{1} t+4 \norm{\nabla w}_{\infty} \norm{\nabla G_{\nu}}_{1} t^{\frac{1}{2}} \right) \norm{\varphi}_{1}.
	\end{align*}
\end{proof}

\section{Proof of Theorem \ref{th:P_weight}} \label{sec:proof_weight}

We prove Theorem \ref{th:P_weight} by induction on $m \in \Z_{>0}$.
We first introduce approximate functions of the monomial weights.
For $j \in \left\{ 1, \ldots, n \right\}$ and $\varepsilon \in \left( 0, 1 \right]$, we define a function $w_{j, \varepsilon} \colon \R^{n} \to \R$ by
\begin{align*}
	w_{j, \varepsilon} \left( x \right) \coloneqq x_{j} e^{- \varepsilon \abs{x}^{2}}, \qquad x= \left( x_{1}, \ldots, x_{n} \right) \in \R^{n}.
\end{align*}
Then, we can see that $w_{j, \varepsilon} \in W^{2, \infty} \left( \R^{n} \right)$ and that
\begin{align*}
	\nabla w_{j, \varepsilon} \left( x \right) &=e^{- \varepsilon \abs{x}^{2}} \left( -2 \varepsilon x_{j} x+e_{j} \right), \\
	\Delta w_{j, \varepsilon} \left( x \right) &=e^{- \varepsilon \abs{x}^{2}} \left( 4 \varepsilon^{2} x_{j} \abs{x}^{2} -2 \left( n+2 \right) \varepsilon x_{j} \right),
\end{align*}
whence follows
\begin{align}
	\label{eq:3.a}
	\norm{\nabla w_{j, \varepsilon}}_{\infty} &\leq 2 \sup_{\rho \geq 0} \rho e^{- \rho} +1 \leq 2, \\
	\label{eq:3.b}
	\norm{\Delta w_{j, \varepsilon}}_{\infty} &\leq 4 \varepsilon^{\frac{1}{2}} \sup_{\rho \geq 0} \rho^{\frac{3}{2}} e^{- \rho} +2 \left( n+2 \right) \varepsilon^{\frac{1}{2}} \sup_{\rho \geq 0} \rho^{\frac{1}{2}} e^{- \rho} \leq \left( n+4 \right) \varepsilon^{\frac{1}{2}}.
\end{align}

Now, we consider the case where $m=1$.
Let $u_{0} \in \left( L_{1}^{1} \cap L^{\infty} \right) \left( \R^{n} \right)$, $t>0$, and $\alpha \in \Z_{\geq 0}^{n}$ with $\abs{\alpha} =1$.
Then, there exists $j \in \left\{ 1, \ldots, n \right\}$ such that $\alpha =e_{j}$.
Here and hereafter, different positive constants independent
of $t$ and $\varepsilon$ are denoted by the same letter $C$.
Multiplying \eqref{I} by $w_{j, \varepsilon}$ yields
\begin{align*}
	w_{j, \varepsilon} u \left( t \right)
	&=w_{j, \varepsilon} e^{t \nu \Delta} u_{0} + \int_{0}^{t} w_{j, \varepsilon} e^{\left( t-s \right) \nu \Delta} f \left( u \left( s \right) \right) ds \\
	&=e^{t \nu \Delta} w_{j, \varepsilon} u_{0} + \left[ w_{j, \varepsilon}, e^{t \nu \Delta} \right] u_{0} \\
	&\hspace{1cm} + \int_{0}^{t} e^{\left( t-s \right) \nu \Delta} w_{j, \varepsilon} f \left( u \left( s \right) \right) ds+ \int_{0}^{t} \bigl[ w_{j, \varepsilon}, e^{\left( t-s \right) \nu \Delta} \bigr] {\,} f \left( u \left( s \right) \right) ds.
\end{align*}
From \eqref{eq:P_decay2}, \eqref{eq:3.a}, \eqref{eq:3.b}, and Lemma \ref{lem:CGL_weight_appro}, we obtain
\begin{align*}
	\norm{w_{j, \varepsilon} u \left( t \right)}_{1} &\leq \norm{e^{t \nu \Delta} w_{j, \varepsilon} u_{0}}_{1} + \norm{\left[ w_{j, \varepsilon}, e^{t \nu \Delta} \right] u_{0}}_{1} \\
	&\hspace{1cm} + \int_{0}^{t} \norm{e^{\left( t-s \right) \nu \Delta} w_{j, \varepsilon} f \left( u \left( s \right) \right)}_{1} ds \\
	&\hspace{1cm} + \int_{0}^{t} \norm{\bigl[ w_{j, \varepsilon}, e^{\left( t-s \right) \nu \Delta} \bigr] {\,} f \left( u \left( s \right) \right)}_{1} ds \\
	&\leq C \norm{w_{j, \varepsilon} u_{0}}_{1} +C \left( \norm{\Delta w_{j, \varepsilon}}_{\infty} t+ \norm{\nabla w_{j, \varepsilon}}_{\infty} t^{\frac{1}{2}} \right) \norm{u_{0}}_{1} \\
	&\hspace{1cm} +C \int_{0}^{t} \norm{u \left( s \right)}_{\infty}^{p-1} \norm{w_{j, \varepsilon} u \left( s \right)}_{1} ds \\
	&\hspace{1cm} +C \int_{0}^{t} \left( \norm{\Delta w_{j, \varepsilon}}_{\infty} \left( t-s \right) + \norm{\nabla w_{j, \varepsilon}}_{\infty} \left( t-s \right)^{\frac{1}{2}} \right) \norm{u \left( s \right)}_{p}^{p} ds \\
	&\leq C \norm{x_{j} u_{0}}_{1} +C \left( \varepsilon^{\frac{1}{2}} t+t^{\frac{1}{2}} \right) \norm{u_{0}}_{1} \\
	&\hspace{1cm} +C \int_{0}^{t} \left( 1+s \right)^{- \frac{n}{2} \left( p-1 \right)} \norm{w_{j, \varepsilon} u \left( s \right)}_{1} ds \\
	&\hspace{1cm} +C \int_{0}^{t} \left( \varepsilon^{\frac{1}{2}} \left( t-s \right) + \left( t-s \right)^{\frac{1}{2}} \right) \left( 1+s \right)^{- \frac{n}{2} \left( p-1 \right)} ds \\
	&\leq \xi_{\varepsilon} \left( t \right) + \int_{0}^{t} \eta \left( s \right) \norm{w_{j, \varepsilon} u \left( s \right)}_{1} ds,
\end{align*}
where
\begin{align*}
	\xi_{\varepsilon} \left( t \right) &\coloneqq C \left( 1+ \varepsilon^{\frac{1}{2}} t+t^{\frac{1}{2}} \right) +C \int_{0}^{t} \left( \varepsilon^{\frac{1}{2}} \left( t-s \right) + \left( t-s \right)^{\frac{1}{2}} \right) \left( 1+s \right)^{- \frac{n}{2} \left( p-1 \right)} ds, \\
	\eta \left( t \right) &\coloneqq C \left( 1+t \right)^{- \frac{n}{2} \left( p-1 \right)}.
\end{align*}
In addition, since $p>1+2/n$, we have
\begin{align*}
	\xi_{\varepsilon} \left( t \right) &\leq C \left( 1+ \varepsilon^{\frac{1}{2}} t+t^{\frac{1}{2}} \right) +C \left( \varepsilon^{\frac{1}{2}} t+t^{\frac{1}{2}} \right) \int_{0}^{t} \left( 1+s \right)^{- \frac{n}{2} \left( p-1 \right)} ds \\
	&\leq C \left( 1+ \varepsilon^{\frac{1}{2}} t+t^{\frac{1}{2}} \right).
\end{align*}
Therefore, it follows from the Gr\"{o}nwall lemma that
\begin{align*}
	\norm{w_{j, \varepsilon} u \left( t \right)}_{1} &\leq \xi_{\varepsilon} \left( t \right) + \int_{0}^{t} \xi_{\varepsilon} \left( s \right) \eta \left( s \right) \exp \left( \int_{s}^{t} \eta \left( \tau \right) d \tau \right) ds,
\end{align*}
and the right hand side on the above inequality is estimated as
\begin{align*}
	&\xi_{\varepsilon} \left( t \right) + \int_{0}^{t} \xi_{\varepsilon} \left( s \right) \eta \left( s \right) \exp \left( \int_{s}^{t} \eta \left( \tau \right) d \tau \right) ds \\
	&\hspace{1cm} \leq C \left( 1+ \varepsilon^{\frac{1}{2}} t+t^{\frac{1}{2}} \right) \left( 1+ \exp \left( C \int_{0}^{+ \infty} \left( 1+ \tau \right)^{- \frac{n}{2} \left( p-1 \right)} d \tau \right) \int_{0}^{t} \left( 1+s \right)^{- \frac{n}{2} \left( p-1 \right)} ds \right) \\
	&\hspace{1cm} \leq C \left( 1+ \varepsilon^{\frac{1}{2}} t+t^{\frac{1}{2}} \right).
\end{align*}
By taking $\varepsilon \searrow 0$ and applying Fatou's lemma, we obtain $x_{j} u \left( t \right) \in L^{1} \left( \R^{n} \right)$ and
\begin{align*}
	\norm{x_{j} u \left( t \right)}_{1} &\leq C \left( 1+t^{\frac{1}{2}} \right).
\end{align*}
This implies $x_{j} f \left( u \left( t \right) \right) \in L^{1} \left( \R^{n} \right)$ in turn.
In addition, by \eqref{I} and Theorem \ref{th:CGL_commutator}, we have
\begin{align*}
	x_{j} u \left( t \right)
	&=x_{j} e^{t \nu \Delta} u_{0} + \int_{0}^{t} x_{j} e^{\left( t-s \right) \nu \Delta} f \left( u \left( s \right) \right) ds \\
	&=e^{t \nu \Delta} x_{j} u_{0} -2t \nu \partial_{j} e^{t \nu \Delta} u_{0} + \int_{0}^{t} e^{\left( t-s \right) \nu \Delta} x_{j} f \left( u \left( s \right) \right) ds-2 \nu \int_{0}^{t} \left( t-s \right) \partial_{j} e^{\left( t-s \right) \nu \Delta} f \left( u \left( s \right) \right) ds,
\end{align*}
whence follows $x_{j} u \in C \left( \left[ 0, + \infty \right); L^{1} \left( \R^{n} \right) \right)$.
This completes the proof of Theorem \ref{th:P_weight} with $m=1$.

Next, we assume that Theorem \ref{th:P_weight} holds for some $m \in \Z_{>0}$.
Let $u_{0} \in \left( L_{m+1}^{1} \cap L^{\infty} \right) \left( \R^{n} \right)$, $t>0$, and $\alpha' \in \Z_{\geq 0}^{n}$ with $\abs{\alpha'} =m+1$.
Then, there exist $\alpha \in \Z_{\geq 0}^{n}$ with $\abs{\alpha} =m$ and $j \in \left\{ 1, \ldots, n \right\}$ such that $\alpha' = \alpha +e_{j}$.
Multiplying \eqref{I} by $w_{j, \varepsilon} x^{\alpha}$ and applying Theorem \ref{th:CGL_commutator}, we obtain
\begin{align*}
	w_{j, \varepsilon} x^{\alpha} u \left( t \right)
	&=w_{j, \varepsilon} x^{\alpha} e^{t \nu \Delta} u_{0} + \int_{0}^{t} w_{j, \varepsilon} x^{\alpha} e^{\left( t-s \right) \nu \Delta} f \left( u \left( s \right) \right) ds \\
	&=e^{t \nu \Delta} w_{j, \varepsilon} x^{\alpha} u_{0} + \left[ w_{j, \varepsilon}, e^{t \nu \Delta} \right] x^{\alpha} u_{0} + w_{j, \varepsilon} R_{\alpha} \left( t \right) u_{0} + \int_{0}^{t} e^{\left( t-s \right) \nu \Delta} w_{j, \varepsilon} x^{\alpha} f \left( u \left( s \right) \right) ds \\
	&\hspace{1cm} + \int_{0}^{t} \bigl[ w_{j, \varepsilon}, e^{\left( t-s \right) \nu \Delta} \bigr] {\,} x ^{\alpha} f \left( u \left( s \right) \right) ds+ \int_{0}^{t} w_{j, \varepsilon} R_{\alpha} \left( t-s \right) f \left( u \left( s \right) \right) ds.
\end{align*}
It follows from the induction hypothesis and Remark \ref{rem:CGL_remainder_esti} that
\begin{align}
	\label{eq:3.c}
	\norm{w_{j, \varepsilon} R_{\alpha} \left( t \right) u_{0}}_{1} &\leq \norm{x_{j} R_{\alpha} \left( t \right) u_{0}}_{1} \nonumber \\
	&\leq C \left\{ t^{\frac{1}{2}} \norm{\abs{x}^{m} u_{0}}_{1} + \left( t^{\frac{1}{2}} +t^{\frac{m+1}{2}} \right) \norm{u_{0}}_{1} \right\}, \\
	\label{eq:3.d}
	\norm{w_{j, \varepsilon} R_{\alpha} \left( t-s \right) f \left( u \left( s \right) \right)}_{1} &\leq \norm{x_{j} R_{\alpha} \left( t-s \right) f \left( u \left( s \right) \right)}_{1} \nonumber \\
	&\leq C \left\{ \left( t-s \right)^{\frac{1}{2}} \norm{\abs{x}^{m} f \left( u \left( s \right) \right)}_{1} + \left( \left( t-s \right)^{\frac{1}{2}} + \left( t-s \right)^{\frac{m+1}{2}} \right) \norm{f \left( u \left( s \right) \right)}_{1} \right\}
\end{align}
for any $s \in \left( 0, t \right)$.
Thus, by a calculation similar to that in the case where $m=1$ with \eqref{eq:P_decay2}, \eqref{eq:3.a}, \eqref{eq:3.b}, \eqref{eq:3.c}, \eqref{eq:3.d}, and Lemma \ref{lem:CGL_weight_appro}, we can derive
\begin{align*}
	\norm{w_{j, \varepsilon} x^{\alpha} u \left( t \right)}_{1} \leq \widetilde{\xi}_{\varepsilon} \left( t \right) + \int_{0}^{t} \eta \left( s \right) \norm{w_{j, \varepsilon} x^{\alpha} u \left( s \right)}_{1} ds,
\end{align*}
where
\begin{align*}
	\widetilde{\xi}_{\varepsilon} \left( t \right) &\coloneqq C \left( 1+t^{\frac{1}{2}} + \varepsilon^{\frac{1}{2}} t+t^{\frac{m+1}{2}} \right) +C \int_{0}^{t} \left( \varepsilon^{\frac{1}{2}} \left( t-s \right) + \left( t-s \right)^{\frac{1}{2}} \right) \left( 1+s \right)^{- \frac{n}{2} \left( p-1 \right)} \left( 1+s^{\frac{m}{2}} \right) ds \\
	&\hspace{1cm} +C \int_{0}^{t} \left( \left( t-s \right)^{\frac{1}{2}} + \left( t-s \right)^{\frac{m+1}{2}} \right) \left( 1+s \right)^{- \frac{n}{2} \left( p-1 \right)} ds.
\end{align*}
Taking into account the assumption that $p>1+2/n$, we have
\begin{align*}
	\widetilde{\xi}_{\varepsilon} \left( t \right) &\leq C \left( 1+t^{\frac{1}{2}} + \varepsilon^{\frac{1}{2}} t+t^{\frac{m+1}{2}} \right) +C \left( \varepsilon^{\frac{1}{2}} t+t^{\frac{1}{2}} \right) \left( 1+t^{\frac{m}{2}} \right) \int_{0}^{t} \left( 1+s \right)^{- \frac{n}{2} \left( p-1 \right)} ds \\
	&\hspace{1cm} +C \left( t^{\frac{1}{2}} +t^{\frac{m+1}{2}} \right) \int_{0}^{t} \left( 1+s \right)^{- \frac{n}{2} \left( p-1 \right)} ds \\
	&\leq C \left( 1+t^{\frac{1}{2}} + \varepsilon^{\frac{1}{2}} t+t^{\frac{m+1}{2}} + \varepsilon^{\frac{1}{2}} t^{\frac{m}{2} +1} \right).
\end{align*}
Hence, it follows from the Gr\"{o}nwall lemma that
\begin{align*}
	\norm{w_{j, \varepsilon} x^{\alpha} u \left( t \right)}_{1} &\leq \widetilde{\xi}_{\varepsilon} \left( t \right) + \int_{0}^{t} \widetilde{\xi}_{\varepsilon} \left( s \right) \eta \left( s \right) \exp \left( \int_{s}^{t} \eta \left( \tau \right) d \tau \right) ds,
\end{align*}
and the right hand side on the above inequality is estimated as
\begin{align*}
	&\widetilde{\xi}_{\varepsilon} \left( t \right) + \int_{0}^{t} \widetilde{\xi}_{\varepsilon} \left( s \right) \eta \left( s \right) \exp \left( \int_{s}^{t} \eta \left( \tau \right) d \tau \right) ds \\
	&\hspace{1cm} \leq C \left( 1+t^{\frac{1}{2}} + \varepsilon^{\frac{1}{2}} t+t^{\frac{m+1}{2}} + \varepsilon^{\frac{1}{2}} t^{\frac{m}{2} +1} \right) \\
	&\hspace{3cm} \times \left( 1+ \exp \left( C \int_{0}^{+ \infty} \left( 1+ \tau \right)^{- \frac{n}{2} \left( p-1 \right)} d \tau \right) \int_{0}^{t} \left( 1+s \right)^{- \frac{n}{2} \left( p-1 \right)} ds \right) \\
	&\hspace{1cm} \leq C \left( 1+t^{\frac{1}{2}} + \varepsilon^{\frac{1}{2}} t+t^{\frac{m+1}{2}} + \varepsilon^{\frac{1}{2}} t^{\frac{m}{2} +1} \right).
\end{align*}
By taking $\varepsilon \searrow 0$ and applying Fatou's lemma, we obtain $x^{\alpha'} u \left( t \right) =x_{j} x^{\alpha} u \left( t \right) \in L^{1} \left( \R^{n} \right)$ and
\begin{align*}
	\lVert x^{\alpha'} u \left( t \right) \rVert_{1} &\leq C \left( 1+t^{\frac{1}{2}} +t^{\frac{m+1}{2}} \right) \leq C \left( 1+t^{\frac{m+1}{2}} \right).
\end{align*}
This implies $x^{\alpha'} f \left( u \left( t \right) \right) \in L^{1} \left( \R^{n} \right)$ in turn.
Moreover, by \eqref{I} and Theorem \ref{th:CGL_commutator}, we obtain
\begin{align*}
	x^{\alpha'} u \left( t \right)
	&=x^{\alpha'} e^{t \nu \Delta} u_{0} + \int_{0}^{t} x^{\alpha'} e^{\left( t-s \right) \nu \Delta} f \left( u \left( s \right) \right) ds \\
	&=e^{t \nu \Delta} x^{\alpha'} u_{0} +R_{\alpha'} \left( t \right) u_{0} + \int_{0}^{t} e^{\left( t-s \right) \nu \Delta} x^{\alpha'} f \left( u \left( s \right) \right) ds+ \int_{0}^{t} R_{\alpha'} \left( t-s \right) f \left( u \left( s \right) \right) ds,
\end{align*}
whence follows $x^{\alpha'} u \in C \left( \left[ 0, + \infty \right); L^{1} \left( \R^{n} \right) \right)$.
This completes the induction argument.
\qed

\section{Proof of Theorem \ref{th:P_asymptotics_lim}} \label{sec:proof_asymptotics_lim}

We first approximate the large time behavior of the global solution to \eqref{P} by that of a linear combination of the CGL semigroup and its derivatives.

\begin{lem} \label{lem:P_appro}
	Let $N \in \Z_{\geq 0}$ and let $p>1+2 \left( N+1 \right) /n$.
	Let $u_{0} \in \left( L^{1} \cap L^{\infty} \right) \left( \R^{n} \right)$ and let $u \in X$ be a global solution to \eqref{P} satisfying \eqref{eq:P_decay}.
	Then, for any $q \in \left[ 1, + \infty \right]$, there exists $C>0$ such that the estimates
	\begin{align*}
		t^{\frac{n}{2} \left( 1- \frac{1}{q} \right)} \norm{u \left( t \right) -e^{t \nu \Delta} u_{0} - \sum_{k=0}^{N} \frac{1}{k!} \left( - \nu \Delta \right)^{k} e^{t \nu \Delta} \psi_{k}}_{q} &\leq \begin{cases}
			Ct^{- \sigma} &\text{if} \quad N< \sigma <N+1, \\
			Ct^{- \left( N+1 \right)} \log \left( 2+t \right) &\text{if} \quad \sigma =N+1, \\
			Ct^{- \left( N+1 \right)} &\text{if} \quad \sigma >N+1
		\end{cases}
	\end{align*}
	hold for all $t>1$, where
	\begin{align*}
		\sigma &\coloneqq \frac{n}{2} \left( p-1 \right) -1>N, \\
		\psi_{k} &\coloneqq \int_{0}^{+ \infty} s^{k} f \left( u \left( s \right) \right) ds.
	\end{align*}
\end{lem}

\begin{rem}
	We do not need the weighted $L^{1}$-spaces in Lemma \ref{lem:P_appro}.
	In addition, we do not need to suppose the regularity of the initial data by virtue of the smoothing effect of the CGL semigroup.
\end{rem}

\begin{proof}[Proof of Lemma \ref{lem:P_appro}]
	For any $k \in \left\{ 0, \ldots, N \right\}$, it follows from \eqref{eq:P_decay2} that
	\begin{align*}
		\norm{\psi_{k}}_{1} + \norm{\psi_{k}}_{\infty}
		&\leq C \int_{0}^{+ \infty} s^{k} \norm{u \left( s \right)}_{\infty}^{p-1} \left( \norm{u \left( s \right)}_{1} + \norm{u \left( s \right)}_{\infty} \right) ds \\
		&\leq C \int_{0}^{+ \infty} \left( 1+s \right)^{k- \frac{n}{2} \left( p-1 \right)} ds \\
		&\leq C,
	\end{align*}
	which implies $\psi_{k} \in \left( L^{1} \cap L^{\infty} \right) \left( \R^{n} \right)$.
	Let $q \in \left[ 1, + \infty \right]$ and let $t>1$.
	We note that the identity
	\begin{align*}
		e^{\left( t-s \right) \nu \Delta} f \left( u \left( s \right) \right)
		&= \sum_{k=0}^{N} \frac{1}{k!} \left( -s \nu \Delta \right)^{k} e^{t \nu \Delta} f \left( u \left( s \right) \right) \\
		&\hspace{1cm} + \frac{1}{N!} \int_{0}^{1} \left( 1- \theta \right)^{N} \left( -s \nu \Delta \right)^{N+1} e^{\left( t-s \theta \right) \nu \Delta} f \left( u \left( s \right) \right) d \theta
	\end{align*}
	holds for any $s \in \left( 0, t \right)$ by virtue of Taylor's theorem with respect to the time variable.
	Using \eqref{I} and the above identity, we decompose the remainder into three parts:
	\begin{align*}
		&u \left( t \right) -e^{t \nu \Delta} u_{0} - \sum_{k=0}^{N} \frac{1}{k!} \left( - \nu \Delta \right)^{k} e^{t \nu \Delta} \psi_{k} \\
		&\hspace{1cm} = \int_{0}^{t/2} \left( e^{\left( t-s \right) \nu \Delta} - \sum_{k=0}^{N} \frac{1}{k!} \left( -s \nu \Delta \right)^{k} e^{t \nu \Delta} \right) f \left( u \left( s \right) \right) ds+ \int_{t/2}^{t} e^{\left( t-s \right) \nu \Delta} f \left( u \left( s \right) \right) ds \\
		&\hspace{2cm} - \sum_{k=0}^{N} \frac{1}{k!} \left( - \nu \Delta \right)^{k} e^{t \nu \Delta} \int_{t/2}^{+ \infty} s^{k} f \left( u \left( s \right) \right) ds \\
		&\hspace{1cm} = \frac{1}{N!} \int_{0}^{t/2} \int_{0}^{1} \left( 1- \theta \right)^{N} \left( -s \nu \Delta \right)^{N+1} e^{\left( t-s \theta \right) \nu \Delta} f \left( u \left( s \right) \right) d \theta ds+ \int_{t/2}^{t} e^{\left( t-s \right) \nu \Delta} f \left( u \left( s \right) \right) ds \\
		&\hspace{2cm} - \sum_{k=0}^{N} \frac{1}{k!} \left( - \nu \Delta \right)^{k} e^{t \nu \Delta} \int_{t/2}^{+ \infty} s^{k} f \left( u \left( s \right) \right) ds.
	\end{align*}
	By \eqref{eq:P_decay} and Lemma \ref{lem:CGL_Lp-Lq}, we have
	\begin{align*}
		\norm{\int_{t/2}^{t} e^{\left( t-s \right) \nu \Delta} f \left( u \left( s \right) \right) ds}_{q}
		&\leq \int_{t/2}^{t} \norm{e^{\left( t-s \right) \nu \Delta} f \left( u \left( s \right) \right)}_{q} ds \\
		&\leq C \int_{t/2}^{t} \norm{u \left( s \right)}_{pq}^{p} ds \\
		&\leq C \int_{t/2}^{t} s^{- \frac{n}{2} \left( 1- \frac{1}{pq} \right) p} ds \\
		&\leq Ct^{- \frac{n}{2} \left( 1- \frac{1}{q} \right) - \sigma}, \\
		\norm{\sum_{k=0}^{N} \frac{1}{k!} \left( - \nu \Delta \right)^{k} e^{t \nu \Delta} \int_{t/2}^{+ \infty} s^{k} f \left( u \left( s \right) \right) ds}_{q}
		&\leq C \sum_{k=0}^{N} t^{- \frac{n}{2} \left( 1- \frac{1}{q} \right) -k} \int_{t/2}^{+ \infty} s^{k} \norm{u \left( s \right)}_{p}^{p} ds \\
		&\leq C \sum_{k=0}^{N} t^{- \frac{n}{2} \left( 1- \frac{1}{q} \right) -k} \int_{t/2}^{+ \infty} s^{k- \frac{n}{2} \left( p-1 \right)} ds \\
		&\leq Ct^{- \frac{n}{2} \left( 1- \frac{1}{q} \right) - \sigma}.
	\end{align*}
	Here, we have used the identity
	\begin{align*}
		\frac{n}{2} \left( 1- \frac{1}{pq} \right) p-1= \frac{n}{2} \left( 1- \frac{1}{q} \right) + \sigma.
	\end{align*}
	In the same way, we obtain
	\begin{align*}
		&\norm{\frac{1}{N!} \int_{0}^{t/2} \int_{0}^{1} \left( 1- \theta \right)^{N} \left( -s \nu \Delta \right)^{N+1} e^{\left( t-s \theta \right) \nu \Delta} f \left( u \left( s \right) \right) d \theta ds}_{q} \\
		&\hspace{1cm} \leq C \int_{0}^{t/2} \int_{0}^{1} s^{N+1} \norm{\Delta^{N+1} e^{\left( t-s \theta \right) \nu \Delta} f \left( u \left( s \right) \right)}_{q} d \theta ds \\
		&\hspace{1cm} \leq C \int_{0}^{t/2} \int_{0}^{1} s^{N+1} \left( t-s \theta \right)^{- \frac{n}{2} \left( 1- \frac{1}{q} \right) - \left( N+1 \right)} \norm{u \left( s \right)}_{p}^{p} d \theta ds \\
		&\hspace{1cm} \leq Ct^{- \frac{n}{2} \left( 1- \frac{1}{q} \right) - \left( N+1 \right)} \int_{0}^{t/2} s^{N+1} \left( 1+s \right)^{- \frac{n}{2} \left( p-1 \right)} ds \\
		&\hspace{1cm} \leq Ct^{- \frac{n}{2} \left( 1- \frac{1}{q} \right) - \left( N+1 \right)} \int_{0}^{t/2} \left( 1+s \right)^{N- \sigma} ds.
	\end{align*}
	The integral appearing on the right hand side of the last inequality is estimated as
	\begin{align*}
		\int_{0}^{t/2} \left( 1+s \right)^{N- \sigma} ds \leq \begin{dcases}
			C \left( 1+ \frac{t}{2} \right)^{N+1- \sigma} \leq Ct^{N+1- \sigma}, &\qquad N< \sigma <N+1, \\
			\log \left( 1+ \frac{t}{2} \right) \leq \log \left( 2+t \right), &\qquad \sigma =N+1, \\
			C, &\qquad \sigma >N+1.
		\end{dcases}
	\end{align*}
	As a consequence, we can deduce the desired estimates.
\end{proof}

We next give the proof of Theorem \ref{th:P_asymptotics_lim}.

\begin{proof}[Proof of Theorem \ref{th:P_asymptotics_lim}]
	Let $q \in \left[ 1, + \infty \right]$ and let $N \coloneqq \left[ m/2 \right] = \max \left\{ j \in \Z_{\geq 0}; {\,} j \leq m/2 \right\}$.
	Then, we see that $m=2N$ or $m=2N+1$ and that
	\begin{align*}
		p &>1+ \frac{m+2}{n} \geq 1+ \frac{2 \left( N+1 \right)}{n}, \\
		\sigma &\coloneqq \frac{n}{2} \left( p-1 \right) -1> \frac{m}{2}.
	\end{align*}
	Therefore, by Lemma \ref{lem:P_appro}, we have
	\begin{align*}
		&t^{\frac{n}{2} \left( 1- \frac{1}{q} \right) + \frac{m}{2}} \norm{u \left( t \right) -e^{t \nu \Delta} u_{0} - \sum_{k=0}^{N} \frac{1}{k!} \left( - \nu \Delta \right)^{k} e^{t \nu \Delta} \psi_{k}}_{q} \\
		&\hspace{1cm} \leq \begin{dcases}
			Ct^{\frac{m}{2} - \sigma} =Ct^{- \left( \sigma - \frac{m}{2} \right)}, &\qquad N< \sigma <N+1, \\
			Ct^{\frac{m}{2} - \left( N+1 \right)} \log \left( 2+t \right) \leq Ct^{- \frac{1}{2}} \log \left( 2+t \right), &\qquad \sigma =N+1, \\
			Ct^{\frac{m}{2} - \left( N+1 \right)} \leq Ct^{- \frac{1}{2}}, &\qquad \sigma >N+1
		\end{dcases}
	\end{align*}
	for any $t>1$, whence follows
	\begin{align}
		\label{eq:4.a}
		\lim_{t \to + \infty} t^{\frac{n}{2} \left( 1- \frac{1}{q} \right) + \frac{m}{2}} \norm{u \left( t \right) -e^{t \nu \Delta} u_{0} - \sum_{k=0}^{N} \frac{1}{k!} \left( - \nu \Delta \right)^{k} e^{t \nu \Delta} \psi_{k}}_{q} =0.
	\end{align}
	Next, let $\gamma \in \Z_{\geq 0}^{n}$ with $\abs{\gamma} \leq N$ and let $\beta \in \Z_{\geq 0}^{n}$ with $\abs{\beta} \leq m-2 \abs{\gamma}$.
	Then, from the assumption that $p>1+ \left( m+2 \right) /n$, we see that
	\begin{align*}
		- \frac{n}{2} \left( p-1 \right) + \frac{\abs{\beta} +2 \abs{\gamma}}{2} <-1.
	\end{align*}
	By \eqref{eq:P_decay2} and Theorem \ref{th:P_weight}, we obtain
	\begin{align*}
		\lVert x^{\beta} \psi_{\abs{\gamma}} \rVert_{1} &\leq \int_{0}^{+ \infty} s^{\abs{\gamma}} {\,} \lVert x^{\beta} f \left( u \left( s \right) \right) \rVert_{1} {\,} ds \\
		&\leq C \int_{0}^{+ \infty} s^{\abs{\gamma}} \norm{u \left( s \right)}_{\infty}^{p-1} \lVert x^{\beta} u \left( s \right) \rVert_{1} {\,} ds \\
		&\leq C \int_{0}^{+ \infty} \left( 1+s \right)^{- \frac{n}{2} \left( p-1 \right) + \frac{\abs{\beta} +2 \abs{\gamma}}{2}} ds \\
		&\leq C,
	\end{align*}
	which implies $\psi_{\abs{\gamma}} \in L^{1}_{m-2 \abs{\gamma}} \left( \R^{n} \right)$.
	Therefore, it follows from Proposition \ref{pro:CGL_asymptotics_lim} that
	\begin{align}
		\label{eq:4.b}
		&\lim_{t \to + \infty} t^{\frac{n}{2} \left( 1- \frac{1}{q} \right) + \frac{m}{2}} \norm{e^{t \nu \Delta} u_{0} - \Lambda_{0, m} \left( t; u_{0} \right)}_{q} =0, \\
		\label{eq:4.c}
		&\lim_{t \to + \infty} t^{\frac{n}{2} \left( 1- \frac{1}{q} \right) + \frac{m}{2}} \norm{\partial^{2 \gamma} e^{t \nu \Delta} \psi_{\abs{\gamma}} - \Lambda_{2 \gamma, m-2 \abs{\gamma}} \left( t; \psi_{\abs{\gamma}} \right)}_{q} =0.
	\end{align}
	Combining \eqref{eq:4.a}, \eqref{eq:4.b}, \eqref{eq:4.c}, and the identity
	\begin{align*}
		\sum_{k=0}^{N} \frac{1}{k!} \left( - \nu \Delta \right)^{k}
		= \sum_{k=0}^{N} \frac{\left( - \nu \right)^{k}}{k!} \left( \sum_{j=1}^{n} \partial_{j}^{2} \right)^{k}
		= \sum_{k=0}^{N} \frac{\left( - \nu \right)^{k}}{k!} \left( \sum_{\abs{\gamma} =k} \frac{k!}{\gamma !} \partial^{2 \gamma} \right)
		= \sum_{\abs{\gamma} \leq N} \frac{\left( - \nu \right)^{\abs{\gamma}}}{\gamma !} \partial^{2 \gamma},
	\end{align*}
	we have
	\begin{align*}
		&\limsup_{t \to + \infty} t^{\frac{n}{2} \left( 1- \frac{1}{q} \right) + \frac{m}{2}} \norm{u \left( t \right) - \AA_{m} \left( t \right)}_{q} \\
		&\hspace{1cm} \leq \limsup_{t \to + \infty} t^{\frac{n}{2} \left( 1- \frac{1}{q} \right) + \frac{m}{2}} \norm{u \left( t \right) -e^{t \nu \Delta} u_{0} - \sum_{k=0}^{N} \frac{1}{k!} \left( - \nu \Delta \right)^{k} e^{t \nu \Delta} \psi_{k}}_{q} \\
		&\hspace{2cm} + \limsup_{t \to + \infty} t^{\frac{n}{2} \left( 1- \frac{1}{q} \right) + \frac{m}{2}} \norm{e^{t \nu \Delta} u_{0} - \Lambda_{0, m} \left( t; u_{0} \right)}_{q} \\
		&\hspace{2cm} + \sum_{\abs{\gamma} \leq N} \frac{\abs{\nu}^{\abs{\gamma}}}{\gamma !} \limsup_{t \to + \infty} t^{\frac{n}{2} \left( 1- \frac{1}{q} \right) + \frac{m}{2}} \norm{\partial^{2 \gamma} e^{t \nu \Delta} \psi_{\abs{\gamma}} - \Lambda_{2 \gamma, m-2 \abs{\gamma}} \left( t; \psi_{\abs{\gamma}} \right)}_{q} \\
		&\hspace{1cm} =0,
	\end{align*}
	whence follows the desired result.
\end{proof}

\section{Proof of Theorem \ref{th:P_asymptotics}} \label{sec:proof_asymptotics}

We show Theorem \ref{th:P_asymptotics} by estimating the difference $u \left( t \right) - \AA_{m} \left( t \right)$ directly without using Lemma \ref{lem:P_appro}.
Let $q \in \left[ 1, + \infty \right]$ and let $t>1$.
We define $N \coloneqq \left[ m/2 \right] = \max \left\{ j \in \Z_{\geq 0}; {\,} j \leq m/2 \right\}$.
Then, we see that $m=2N$ or $m=2N+1$ and that
\begin{align*}
	p &>1+ \frac{m+2}{n} \geq 1+ \frac{2 \left( N+1 \right)}{n}, \\
	\sigma &\coloneqq \frac{n}{2} \left( p-1 \right) -1> \frac{m}{2} \geq N.
\end{align*}
By virtue of Taylor's theorem with respect to the time variable, we obtain
\begin{align*}
	e^{\left( t-s \right) \nu \Delta} f \left( u \left( s \right) \right) &= \sum_{k=0}^{N} \frac{1}{k!} \left( -s \nu \Delta \right)^{k} e^{t \nu \Delta} f \left( u \left( s \right) \right) \\
	&\hspace{1cm} + \frac{1}{N!} \int_{0}^{1} \left( 1- \theta \right)^{N} \left( -s \nu \Delta \right)^{N+1} e^{\left( t-s \theta \right) \nu \Delta} f \left( u \left( s \right) \right) d \theta \\
	&= \sum_{\abs{\gamma} \leq N} \frac{\left( - \nu \right)^{\abs{\gamma}}}{\gamma !} s^{\abs{\gamma}} \partial^{2 \gamma} e^{t \nu \Delta} f \left( u \left( s \right) \right) \\
	&\hspace{1cm} + \left( N+1 \right) \sum_{\abs{\gamma} =N+1} \frac{\left( - \nu \right)^{N+1}}{\gamma !} \int_{0}^{1} \left( 1- \theta \right)^{N} s^{N+1} \partial^{2 \gamma} e^{\left( t-s \theta \right) \nu \Delta} f \left( u \left( s \right) \right) d \theta.
\end{align*}
Using \eqref{I} and the above identity, we decompose the remainder into five parts:
\begin{align*}
	&u \left( t \right) - \AA_{m} \left( t \right) \\
	&\hspace{0.5cm} = \left( e^{t \nu \Delta} u_{0} - \Lambda_{0, m} \left( t; u_{0} \right) \right) + \int_{t/2}^{t} e^{\left( t-s \right) \nu \Delta} f \left( u \left( s \right) \right) ds \\
	&\hspace{1.5cm} + \sum_{\abs{\gamma} \leq N} \frac{\left( - \nu \right)^{\abs{\gamma}}}{\gamma !} \int_{0}^{t/2} s^{\abs{\gamma}} \left( \partial^{2 \gamma} e^{t \nu \Delta} f \left( u \left( s \right) \right) - \Lambda_{2 \gamma, m-2 \abs{\gamma}} \left( t; f \left( u \left( s \right) \right) \right) \right) ds \\
	&\hspace{1.5cm} + \left( N+1 \right) \sum_{\abs{\gamma} =N+1} \frac{\left( - \nu \right)^{N+1}}{\gamma !} \int_{0}^{t/2} \int_{0}^{1} \left( 1- \theta \right)^{N} s^{N+1} \partial^{2 \gamma} e^{\left( t-s \theta \right) \nu \Delta} f \left( u \left( s \right) \right) d \theta ds \\
	&\hspace{1.5cm} - \sum_{\abs{\gamma} \leq N} \frac{\left( - \nu \right)^{\abs{\gamma}}}{\gamma !} \sum_{\abs{\beta} \leq m-2 \abs{\gamma}} 2^{- \left( \abs{\beta} +2 \abs{\gamma} \right)} t^{- \frac{\abs{\beta} +2 \abs{\gamma}}{2}} \left( \int_{t/2}^{+ \infty} s^{\abs{\gamma}} \MM_{\beta} \left( f \left( u \left( s \right) \right) \right) ds \right) \delta_{t} \left( \bm{h}_{\nu, 2 \gamma + \beta} G_{\nu} \right) \\
	&\hspace{0.5cm} \eqqcolon \sum_{\ell =1}^{5} J_{\ell} \left( t \right).
\end{align*}
Firstly, due to \eqref{eq:P_decay2} and Lemma \ref{lem:CGL_Lp-Lq}, we have
\begin{align*}
	\norm{J_{2} \left( t \right)}_{q}
	&\leq \int_{t/2}^{t} \norm{e^{\left( t-s \right) \nu \Delta} f \left( u \left( s \right) \right)}_{q} ds \leq C \int_{t/2}^{t} \norm{u \left( s \right)}_{pq}^{p} ds
	\\
	&\leq C \int_{t/2}^{t} s^{- \frac{n}{2} \left( 1- \frac{1}{pq} \right) p} ds \leq Ct^{- \frac{n}{2} \left( 1- \frac{1}{q} \right) - \sigma}, \\
	\norm{J_{4} \left( t \right)}_{q}
	&\leq C \sum_{\abs{\gamma} =N+1} \int_{0}^{t/2} \int_{0}^{1} s^{N+1} \norm{\partial^{2 \gamma} e^{\left( t-s \theta \right) \nu \Delta} f \left( u \left( s \right) \right)}_{q} d \theta ds \\
	&\leq C \sum_{\abs{\gamma} =N+1} \int_{0}^{t/2} \int_{0}^{1} s^{N+1} \left( t-s \theta \right)^{- \frac{n}{2} \left( 1- \frac{1}{q} \right) - \abs{\gamma}} \norm{u \left( s \right)}_{p}^{p} d \theta ds \\
	&\leq C \sum_{\abs{\gamma} =N+1} t^{- \frac{n}{2} \left( 1- \frac{1}{q} \right) - \abs{\gamma}} \int_{0}^{t/2} s^{N+1} \left( 1+s \right)^{- \frac{n}{2} \left( p-1 \right)} ds \\
	&\leq Ct^{- \frac{n}{2} \left( 1- \frac{1}{q} \right) - \left( N+1 \right)} \int_{0}^{t/2} \left( 1+s \right)^{N- \sigma} ds \\
	&\leq Ct^{- \frac{n}{2} \left( 1- \frac{1}{q} \right) - \left( N+1 \right)} \times \begin{cases}
		t^{N+1- \sigma}, &\qquad N< \sigma <N+1, \\
		\log \left( 2+t \right), &\qquad \sigma =N+1, \\
		1, &\qquad \sigma >N+1
	\end{cases} \\
	&\leq Ct^{- \frac{n}{2} \left( 1- \frac{1}{q} \right) - \frac{m}{2}} \times \begin{dcases}
		t^{- \left( \sigma - \frac{m}{2} \right)}, &\qquad 1+ \frac{m+2}{n} <p<1+ \frac{m+3}{n}, \\
		t^{- \frac{1}{2}} \log \left( 2+t \right), &\qquad p=1+ \frac{m+3}{n}, \\
		t^{- \frac{1}{2}}, &\qquad p>1+ \frac{m+3}{n}.
	\end{dcases}
\end{align*}
Secondly, it follows from \eqref{eq:P_decay2} and Theorem \ref{th:P_weight} that
\begin{align*}
	\norm{J_{5} \left( t \right)}_{q}
	&\leq C \sum_{\abs{\gamma} \leq N} \sum_{\abs{\beta} \leq m-2 \abs{\gamma}} t^{- \frac{\abs{\beta} +2 \abs{\gamma}}{2}} \left( \int_{t/2}^{+ \infty} s^{\abs{\gamma}} \abs{\MM_{\beta} \left( f \left( u \left( s \right) \right) \right)} ds \right) \norm{\delta_{t} \left( \bm{h}_{\nu, 2 \gamma + \beta} G_{\nu} \right)}_{q} \\
	&\leq C \sum_{\abs{\gamma} \leq N} \sum_{\abs{\beta} \leq m-2 \abs{\gamma}} t^{- \frac{n}{2} \left( 1- \frac{1}{q} \right) - \frac{\abs{\beta} +2 \abs{\gamma}}{2}} \int_{t/2}^{+ \infty} s^{\abs{\gamma}} \norm{u \left( s \right)}_{\infty}^{p-1} \lVert x^{\beta} u \left( s \right) \rVert_{1} {\,} ds \\
	&\leq C \sum_{\abs{\gamma} \leq N} \sum_{\abs{\beta} \leq m-2 \abs{\gamma}} t^{- \frac{n}{2} \left( 1- \frac{1}{q} \right) - \frac{\abs{\beta} +2 \abs{\gamma}}{2}} \int_{t/2}^{+ \infty} \left( 1+s \right)^{- \frac{n}{2} \left( p-1 \right) + \frac{\abs{\beta} + 2 \abs{\gamma}}{2}} ds \\
	&\leq C \sum_{\abs{\gamma} \leq N} \sum_{\abs{\beta} \leq m-2 \abs{\gamma}} t^{- \frac{n}{2} \left( 1- \frac{1}{q} \right) - \frac{\abs{\beta} +2 \abs{\gamma}}{2}} \left( 1+t \right)^{1- \frac{n}{2} \left( p-1 \right) + \frac{\abs{\beta} + 2 \abs{\gamma}}{2}} \\
	&\leq Ct^{- \frac{n}{2} \left( 1- \frac{1}{q} \right) - \sigma}.
\end{align*}
Here, we have used the fact that
\begin{align*}
	p>1+ \frac{m+2}{n} \quad \Longleftrightarrow \quad - \frac{n}{2} \left( p-1 \right) + \frac{m}{2} <-1.
\end{align*}
Finally, by \eqref{eq:P_decay2}, Theorem \ref{th:P_weight}, and Proposition \ref{pro:CGL_asymptotics}, we obtain
\begin{align*}
	\norm{J_{1} \left( t \right)}_{q} &\leq Ct^{- \frac{n}{2} \left( 1- \frac{1}{q} \right) - \frac{m+1}{2}} \sum_{\abs{\alpha} =m+1} \norm{x^{\alpha} u_{0}}_{1}, \\
	\norm{J_{3} \left( t \right)}_{q}
	&\leq Ct^{- \frac{n}{2} \left( 1- \frac{1}{q} \right) - \frac{m+1}{2}} \sum_{\abs{\gamma} \leq N} \sum_{\abs{\beta} =m-2 \abs{\gamma} +1} \int_{0}^{t/2} s^{\abs{\gamma}} \norm{u \left( s \right)}_{\infty}^{p-1} \lVert x^{\beta} u \left( s \right) \rVert_{1} {\,} ds \\
	&\leq Ct^{- \frac{n}{2} \left( 1- \frac{1}{q} \right) - \frac{m+1}{2}} \sum_{\abs{\gamma} \leq N} \sum_{\abs{\beta} =m-2 \abs{\gamma} +1} \int_{0}^{t/2} \left( 1+s \right)^{- \frac{n}{2} \left( p-1 \right) + \frac{\abs{\beta} +2 \abs{\gamma}}{2}} ds \\
	&\leq Ct^{- \frac{n}{2} \left( 1- \frac{1}{q} \right) - \frac{m+1}{2}} \int_{0}^{t/2} \left( 1+s \right)^{- \frac{n}{2} \left( p-1 \right) + \frac{m+1}{2}} ds \\
	&\leq Ct^{- \frac{n}{2} \left( 1- \frac{1}{q} \right) - \frac{m+1}{2}} \times \begin{dcases}
		\left( 1+t \right)^{1- \frac{n}{2} \left( p-1 \right) + \frac{m+1}{2}}, &\qquad p<1+ \frac{m+3}{n}, \\
		\log \left( 2+t \right), &\qquad p=1+ \frac{m+3}{n}, \\
		1, &\qquad p>1+ \frac{m+3}{n}
	\end{dcases} \\
	&\leq Ct^{- \frac{n}{2} \left( 1- \frac{1}{q} \right) - \frac{m}{2}} \times \begin{dcases}
		t^{- \left( \sigma - \frac{m}{2} \right)}, &\qquad p<1+ \frac{m+3}{n}, \\
		t^{- \frac{1}{2}} \log \left( 2+t \right), &\qquad p=1+ \frac{m+3}{n}, \\
		t^{- \frac{1}{2}}, &\qquad p>1+ \frac{m+3}{n}.
	\end{dcases}
\end{align*}
Combining these estimates, we conclude that
\begin{align*}
	t^{\frac{n}{2} \left( 1- \frac{1}{q} \right) + \frac{m}{2}} \norm{u \left( t \right) - \AA_{m} \left( t \right)}_{q} &\leq \begin{dcases}
		Ct^{- \left( \sigma - \frac{m}{2} \right)}, &\qquad 1+ \frac{m+2}{n} <p<1+ \frac{m+3}{n}, \\
		Ct^{- \frac{1}{2}} \log \left( 2+t \right), &\qquad p=1+ \frac{m+3}{n}, \\
		Ct^{- \frac{1}{2}}, &\qquad p>1+ \frac{m+3}{n}
	\end{dcases}
\end{align*}
for all $t>1$.
On the other hand, by a simple calculation, we have
\begin{align*}
	t^{\frac{n}{2} \left( 1- \frac{1}{q} \right) + \frac{m}{2}} \norm{u \left( t \right) - \AA_{m} \left( t \right)}_{q} \leq t^{\frac{n}{2} \left( 1- \frac{1}{q} \right) + \frac{m}{2}} \left( \norm{u \left( t \right)}_{q} + \norm{\AA_{m} \left( t \right)}_{q} \right) \leq C
\end{align*}
for all $t \in \left( 0, 1 \right]$.
As a consequence, we arrive at the desired estimates.
\qed

\begin{rem}
	By using Lemma \ref{lem:P_appro}, we can prove Theorem \ref{th:P_asymptotics} in the case where $p>1+ \left( m+3 \right) /n$, while we cannot prove Theorem \ref{th:P_asymptotics} in the case where $1+ \left( m+2 \right) /n<p \leq 1+ \left( m+3 \right) /n$.
	This implies that the approximation given in Lemma \ref{lem:P_appro} is too rough to derive the principal part near $t=0$.
\end{rem}

\section{Proof of Theorem \ref{th:P_asymptotics_optimal}} \label{sec:proof_optimal}

Before the proof of Theorem \ref{th:P_asymptotics_optimal}, we prepare the following lemma which describes the structure of the asymptotic profiles given in Theorems \ref{th:P_asymptotics_lim} and \ref{th:P_asymptotics}.

\begin{lem} \label{lem:P_Lambda}
	Under the same assumption and notation as in Theorem \ref{th:P_asymptotics_lim}, the identity
	\begin{align*}
		\AA_{k} \left( t \right) - \AA_{k-1} \left( t \right) =t^{- \frac{k}{2}} \delta_{t} \left( \AA_{k} \left( 1 \right) - \AA_{k-1} \left( 1 \right) \right)
	\end{align*}
	holds for any $t>0$ and $k \in \left\{ 0, \ldots, m \right\}$, where $\AA_{-1} \left( t \right) \equiv 0$.
	Moreover, $\AA_{k} \left( 1 \right) - \AA_{k-1} \left( 1 \right)$ is represented explicitly as
	\begin{align*}
		\AA_{k} \left( 1 \right) - \AA_{k-1} \left( 1 \right) =2^{-k} \sum_{\abs{\alpha} =k} {\,} \Biggl( \MM_{\alpha} \left( u_{0} \right) + \sum_{\substack{\beta +2 \gamma = \alpha \\ \abs{\gamma} \leq k/2}} \frac{\left( - \nu \right)^{\abs{\gamma}}}{\gamma !} \MM_{\beta} \left( \psi_{\abs{\gamma}} \right) \Biggr) {\,} \bm{h}_{\nu, \alpha} G_{\nu}.
	\end{align*}
\end{lem}

\begin{proof}
	The case where $k=0$ follows from $\bm{h}_{\nu, 0} \equiv 1$ and
	\begin{align*}
		\AA_{0} \left( t \right) = \left( \MM_{0} \left( u_{0} \right) + \MM_{0} \left( \psi_{0} \right) \right) \delta_{t} G_{\nu}.
	\end{align*}
	We next show the case where $k \in \left\{ 1, \ldots, m \right\}$.
	Let $t>0$ and let
	\begin{align*}
		N \coloneqq \left[ \frac{k-1}{2} \right] = \max \left\{ j \in \Z_{\geq 0}; {\,} j \leq \frac{k-1}{2} \right\}.
	\end{align*}
	Then, we see that $k=2N+1$ or $k=2 \left( N+1 \right)$ and that
	\begin{align*}
		\left[ \frac{k}{2} \right] = \begin{cases}
			N, &\qquad k=2N+1, \\
			N+1, &\qquad k=2 \left( N+1 \right).
		\end{cases}
	\end{align*}
	For the case where $k=2N+1$, we obtain
	\begin{align*}
		\AA_{k} \left( t \right) - \AA_{k-1} \left( t \right) &= \left( \Lambda_{0, k} \left( t; u_{0} \right) - \Lambda_{0, k-1} \left( t; u_{0} \right) \right) \\
		&\hspace{1cm} + \sum_{\abs{\gamma} \leq N} \frac{\left( - \nu \right)^{\abs{\gamma}}}{\gamma !} \left( \Lambda_{2 \gamma, k-2 \abs{\gamma}} \left( t; \psi_{\abs{\gamma}} \right) - \Lambda_{2 \gamma, k-1-2 \abs{\gamma}} \left( t; \psi_{\abs{\gamma}} \right) \right) \\
		&=2^{-k} t^{- \frac{k}{2}} \sum_{\abs{\alpha} =k} \MM_{\alpha} \left( u_{0} \right) \delta_{t} \left( \bm{h}_{\nu, \alpha} G_{\nu} \right) \\
		&\hspace{1cm} +2^{-k} t^{- \frac{k}{2}} \sum_{\abs{\gamma} \leq N} \frac{\left( - \nu \right)^{\abs{\gamma}}}{\gamma !} \sum_{\abs{\beta} =k-2 \abs{\gamma}} \MM_{\beta} \left( \psi_{\abs{\gamma}} \right) \delta_{t} \left( \bm{h}_{\nu, 2 \gamma + \beta} G_{\nu} \right).
	\end{align*}
	In the same way, for the case where $k=2 \left( N+1 \right)$, we have
	\begin{align*}
		\AA_{k} \left( t \right) - \AA_{k-1} \left( t \right)
		&= \left( \Lambda_{0, k} \left( t; u_{0} \right) - \Lambda_{0, k-1} \left( t; u_{0} \right) \right) \\
		&\hspace{1cm} + \sum_{\abs{\gamma} \leq N} \frac{\left( - \nu \right)^{\abs{\gamma}}}{\gamma !} \left( \Lambda_{2 \gamma, k-2 \abs{\gamma}} \left( t; \psi_{\abs{\gamma}} \right) - \Lambda_{2 \gamma, k-1-2 \abs{\gamma}} \left( t; \psi_{\abs{\gamma}} \right) \right) \\
		&\hspace{1cm} + \sum_{\abs{\gamma} =N+1} \frac{\left( - \nu \right)^{\abs{\gamma}}}{\gamma !} \Lambda_{2 \gamma, k-2 \abs{\gamma}} \left( t; \psi_{\abs{\gamma}} \right) \\
		&=2^{-k} t^{- \frac{k}{2}} \sum_{\abs{\alpha} =k} \MM_{\alpha} \left( u_{0} \right) \delta_{t} \left( \bm{h}_{\nu, \alpha} G_{\nu} \right) \\
		&\hspace{1cm} +2^{-k} t^{- \frac{k}{2}} \sum_{\abs{\gamma} \leq N} \frac{\left( - \nu \right)^{\abs{\gamma}}}{\gamma !} \sum_{\abs{\beta} =k-2 \abs{\gamma}} \MM_{\beta} \left( \psi_{\abs{\gamma}} \right) \delta_{t} \left( \bm{h}_{\nu, 2 \gamma + \beta} G_{\nu} \right) \\
		&\hspace{1cm} +2^{-2 \left( N+1 \right)} t^{- \left( N+1 \right)} \sum_{\abs{\gamma} =N+1} \frac{\left( - \nu \right)^{\abs{\gamma}}}{\gamma !} \MM_{0} \left( \psi_{\abs{\gamma}} \right) \delta_{t} \left( \bm{h}_{\nu, 2 \gamma} G_{\nu} \right) \\
		&=2^{-k} t^{- \frac{k}{2}} \sum_{\abs{\alpha} =k} \MM_{\alpha} \left( u_{0} \right) \delta_{t} \left( \bm{h}_{\nu, \alpha} G_{\nu} \right) \\
		&\hspace{1cm} +2^{-k} t^{- \frac{k}{2}} \sum_{\abs{\gamma} \leq N+1} \frac{\left( - \nu \right)^{\abs{\gamma}}}{\gamma !} \sum_{\abs{\beta} =k-2 \abs{\gamma}} \MM_{\beta} \left( \psi_{\abs{\gamma}} \right) \delta_{t} \left( \bm{h}_{\nu, 2 \gamma + \beta} G_{\nu} \right).
	\end{align*}
	In any case, since the dilation $\delta_{t}$ is linear, $\AA_{k} \left( t \right) - \AA_{k-1} \left( t \right)$ is represented as
	\begin{align*}
		&\AA_{k} \left( t \right) - \AA_{k-1} \left( t \right) \\
		&\hspace{1cm} =2^{-k} t^{- \frac{k}{2}} \delta_{t} \left( \sum_{\abs{\alpha} =k} \MM_{\alpha} \left( u_{0} \right) \bm{h}_{\nu, \alpha} G_{\nu} + \sum_{\abs{\gamma} \leq k/2} \frac{\left( - \nu \right)^{\abs{\gamma}}}{\gamma !} \sum_{\abs{\beta} =k-2 \abs{\gamma}} \MM_{\beta} \left( \psi_{\abs{\gamma}} \right) \bm{h}_{\nu, 2 \gamma + \beta} G_{\nu} \right).
	\end{align*}
	In particular, substituting $t=1$ to the above identity yields
	\begin{align*}
		&\AA_{k} \left( 1 \right) - \AA_{k-1} \left( 1 \right) \\
		&\hspace{1cm} =2^{-k} \left( \sum_{\abs{\alpha} =k} \MM_{\alpha} \left( u_{0} \right) \bm{h}_{\nu, \alpha} G_{\nu} + \sum_{\abs{\gamma} \leq k/2} \frac{\left( - \nu \right)^{\abs{\gamma}}}{\gamma !} \sum_{\abs{\beta} =k-2 \abs{\gamma}} \MM_{\beta} \left( \psi_{\abs{\gamma}} \right) \bm{h}_{\nu, 2 \gamma + \beta} G_{\nu} \right) \\
		&\hspace{1cm} =2^{-k} \sum_{\abs{\alpha} =k} {\,} \Biggl( \MM_{\alpha} \left( u_{0} \right) + \sum_{\substack{\beta +2 \gamma = \alpha \\ \abs{\gamma} \leq k/2}} \frac{\left( - \nu \right)^{\abs{\gamma}}}{\gamma !} \MM_{\beta} \left( \psi_{\abs{\gamma}} \right) \Biggr) {\,} \bm{h}_{\nu, \alpha} G_{\nu}.
	\end{align*}
	This completes the proof.
\end{proof}

Now, we are ready to show Theorem \ref{th:P_asymptotics_optimal}.

\begin{proof}[Proof of Theorem \ref{th:P_asymptotics_optimal}]
	By Lemma \ref{lem:P_Lambda}, we obtain
	\begin{align*}
		u \left( t \right) - \AA_{m} \left( t \right)
		&=u \left( t \right) - \AA_{m+1} \left( t \right) + \left( \AA_{m+1} \left( t \right) - \AA_{m} \left( t \right) \right) \\
		&=u \left( t \right) - \AA_{m+1} \left( t \right) +t^{- \frac{m+1}{2}} \delta_{t} \left( \AA_{m+1} \left( 1 \right) - \AA_{m} \left( 1 \right) \right)
	\end{align*}
	for any $t>0$.
	Due to Theorem \ref{th:P_asymptotics_lim}, we have
	\begin{align*}
		\limsup_{t \to + \infty} t^{\frac{n}{2} \left( 1- \frac{1}{q} \right) + \frac{m}{2} + \frac{1}{2}} \norm{u \left( t \right) - \AA_{m} \left( t \right)}_{q}
		&\leq \limsup_{t \to + \infty} t^{\frac{n}{2} \left( 1- \frac{1}{q} \right) + \frac{m}{2} + \frac{1}{2}} \norm{u \left( t \right) - \AA_{m+1} \left( t \right)}_{q} \\
		&\hspace{1cm} + \limsup_{t \to + \infty} t^{\frac{n}{2} \left( 1- \frac{1}{q} \right)} \norm{\delta_{t} \left( \AA_{m+1} \left( 1 \right) - \AA_{m} \left( 1 \right) \right)}_{q} \\
		&= \norm{\AA_{m+1} \left( 1 \right) - \AA_{m} \left( 1 \right)}_{q}, \\
		\liminf_{t \to + \infty} t^{\frac{n}{2} \left( 1- \frac{1}{q} \right) + \frac{m}{2} + \frac{1}{2}} \norm{u \left( t \right) - \AA_{m} \left( t \right)}_{q}
		&\geq \liminf_{t \to + \infty} t^{\frac{n}{2} \left( 1- \frac{1}{q} \right)} \norm{\delta_{t} \left( \AA_{m+1} \left( 1 \right) - \AA_{m} \left( 1 \right) \right)}_{q} \\
		&\hspace{1cm} - \limsup_{t \to + \infty} t^{\frac{n}{2} \left( 1- \frac{1}{q} \right) + \frac{m}{2} + \frac{1}{2}} \norm{u \left( t \right) - \AA_{m+1} \left( t \right)}_{q} \\
		&= \norm{\AA_{m+1} \left( 1 \right) - \AA_{m} \left( 1 \right)}_{q}
	\end{align*}
	for any $q \in \left[ 1, + \infty \right]$, which implies
	\begin{align*}
		\lim_{t \to + \infty} t^{\frac{n}{2} \left( 1- \frac{1}{q} \right) + \frac{m}{2} + \frac{1}{2}} \norm{u \left( t \right) - \AA_{m} \left( t \right)}_{q} = \norm{\AA_{m+1} \left( 1 \right) - \AA_{m} \left( 1 \right)}_{q}.
	\end{align*}
	The equivalence $\text{(i)} \Leftrightarrow \text{(ii)}$ follows from the orthogonality of the Hermite polynomials, namely,
	\begin{align*}
		\int_{\R^{n}} \bm{h}_{\nu, \alpha} \left( x \right) \bm{h}_{\nu, \beta} \left( x \right) G_{\nu} \left( x \right) dx= \begin{dcases}
			\left( \frac{2}{\nu} \right)^{\abs{\alpha}} \alpha !, &\qquad \alpha = \beta, \\
			0, &\qquad \alpha \neq \beta.
		\end{dcases}
	\end{align*}
	The implication $\text{(i)} \Rightarrow \text{(iii)}$ is a paraphrase of \eqref{eq:P_asymptotics_optimal} with $\norm{\AA_{m+1} \left( 1 \right) - \AA_{m} \left( 1 \right)}_{q} >0$, and the implication $\text{(iii)} \Rightarrow \text{(i)}$ is clear.
\end{proof}

\appendix
\section{Asymptotics for small amplitude solutions} \label{sec:appendix}
In this appendix, we give a sufficient condition for the initial data to ensure $\AA_{m+1} \left( 1 \right) - \AA_{m} \left( 1 \right) \not\equiv 0$.
Our goal can be stated as follows.

\begin{prop} \label{pro:P_small_optimal}
	Let $m \in \Z_{\geq 0}$ and let $p>1+ \left( m+3 \right) /n$.
	Let $\varphi \in \left( L^{1}_{m+1} \cap L^{\infty} \right) \left( \R^{n} \right)$ satisfy $\MM_{\alpha} \left( \varphi \right) \neq 0$ for some $\alpha \in \Z_{\geq 0}^{n}$ with $\abs{\alpha} =m+1$.
	Then, there exists $\varepsilon_{\varphi} >0$ such that for any $\varepsilon \in \left( 0, \varepsilon_{\varphi} \right]$, a global solution $u \in X$ to \eqref{P} with $u_{0} = \varepsilon \varphi$ satisfies $\AA_{m+1} \left( 1 \right) - \AA_{m} \left( 1 \right) \not\equiv 0$.
\end{prop}

To show this proposition, we prepare some lemmas.

\begin{lem} \label{lem:P_SDGE}
	Let $p>1+2/n$.
	Then, there exists $\varepsilon_{0} >0$ such that for any $u_{0} \in \left( L^{1} \cap L^{\infty} \right) \left( \R^{n} \right)$ with $\norm{u_{0}}_{1} + \norm{u_{0}}_{\infty} \leq \varepsilon_{0}$, \eqref{P} has a unique global solution $u \in X$, which satisfies
	\begin{align}
		\label{eq:P_decay_small}
		M \coloneqq \sup_{q \in \left[ 1, + \infty \right]} \sup_{t>0} \left( 1+t \right)^{\frac{n}{2} \left( 1- \frac{1}{q} \right)} \norm{u \left( t \right)}_{q} \leq C \left( \norm{u_{0}}_{1} + \norm{u_{0}}_{\infty} \right)
	\end{align}
	for some $C>0$ independent of $u_{0}$.
\end{lem}

Since we can prove the above lemma by the standard contraction argument to the integral equation \eqref{I}, we omit the proof.

\begin{lem} \label{lem:P_weight_small}
	Let $p>1+2/n$ and let $m \in \Z_{>0}$.
	Let $u_{0} \in \left( L_{m}^{1} \cap L^{\infty} \right) \left( \R^{n} \right)$ satisfy $\norm{u_{0}}_{1} + \norm{u_{0}}_{\infty} \leq \varepsilon_{0}$ and let $u \in X$ be the global solution to \eqref{P} given in Lemma \ref{lem:P_SDGE}.
	Then,
	\begin{align}
		\label{eq:P_weight_small}
		\sup_{t>0} \left( 1+t \right)^{- \frac{m}{2}} \norm{\abs{x}^{m} u \left( t \right)}_{1} \leq C \left( \norm{\abs{x}^{m} u_{0}}_{1} + \norm{u_{0}}_{1} + \norm{u_{0}}_{\infty} \right)
	\end{align}
	for some $C>0$ independent of $u_{0}$.
\end{lem}

\begin{proof}
	Let $C$ denote a positive constant independent of $u_{0}$ and $t$ which may change line to line.
	It follows from Theorem \ref{th:P_weight} that $u \in C \left( \left[ 0, + \infty \right); L^{1}_{m} \left( \R^{n} \right) \right)$.
	By using \eqref{I}, \eqref{eq:CGL_weight}, and \eqref{eq:P_decay_small}, we have
	\begin{align*}
		\left( 1+t \right)^{- \frac{m}{2}} \norm{\abs{x}^{m} u \left( t \right)}_{1}
		&\leq \left( 1+t \right)^{- \frac{m}{2}} \norm{\abs{x}^{m} e^{t \nu \Delta} u_{0}}_{1} + \left( 1+t \right)^{- \frac{m}{2}} \int_{0}^{t} \norm{\abs{x}^{m} e^{\left( t-s \right) \nu \Delta} f \left( u \left( s \right) \right)}_{1} ds \\
		&\leq C \left( 1+t \right)^{- \frac{m}{2}} \left( \norm{\abs{x}^{m} u_{0}}_{1} +t^{\frac{m}{2}} \norm{u_{0}}_{1} \right) \\
		&\hspace{1cm} +C \left( 1+t \right)^{- \frac{m}{2}} \int_{0}^{t} \norm{u \left( s \right)}_{\infty}^{p-1} \left( \norm{\abs{x}^{m} u \left( s \right)}_{1} + \left( t-s \right)^{\frac{m}{2}} \norm{u \left( s \right)}_{1} \right) ds \\
		&\leq C \left( \norm{\abs{x}^{m} u_{0}}_{1} + \norm{u_{0}}_{1} \right) \\
		&\hspace{1cm} +CM^{p-1} \int_{0}^{t} \left( 1+s \right)^{- \frac{n}{2} \left( p-1 \right)} \left( \left( 1+s \right)^{- \frac{m}{2}} \norm{\abs{x}^{m} u \left( s \right)}_{1} +M \right) ds
	\end{align*}
	for any $t>0$.
	Applying the Gr\"{o}nwall lemma yields
	\begin{align*}
		\left( 1+t \right)^{- \frac{m}{2}} \norm{\abs{x}^{m} u \left( t \right)}_{1}
		&\leq C \left( \norm{\abs{x}^{m} u_{0}}_{1} + \norm{u_{0}}_{1} +M \right) \exp \left( CM^{p-1} \int_{0}^{t} \left( 1+s \right)^{- \frac{n}{2} \left( p-1 \right)} ds \right) \\
		&\leq C \left( \norm{\abs{x}^{m} u_{0}}_{1} + \norm{u_{0}}_{1} + \norm{u_{0}}_{\infty} \right) \exp \left( C \varepsilon_{0}^{p-1} \right).
	\end{align*}
\end{proof}

Now, we show Proposition \ref{pro:P_small_optimal}.

\begin{proof}[Proof of Proposition \ref{pro:P_small_optimal}]
	By Lemmas \ref{lem:P_SDGE} and \ref{lem:P_weight_small}, and the uniqueness of global solutions to \eqref{P} in $X$, $u$ satisfies \eqref{eq:P_decay_small} and \eqref{eq:P_weight_small} where $m$ is replaced by $m+1$ provided that $\norm{u_{0}}_{1} + \norm{u_{0}}_{\infty} \leq \varepsilon_{0}$ with $u_{0} = \varepsilon \varphi$, namely, $\varepsilon \leq \varepsilon_{0} / \left( \norm{\varphi}_{1} + \norm{\varphi}_{\infty} \right)$.
	Moreover, we have
	\begin{align*}
		\sum_{\substack{\beta +2 \gamma = \alpha \\ \abs{\gamma} \leq \left( m+1 \right) /2}} \frac{\abs{\nu}^{\abs{\gamma}}}{\gamma !} \abs{\MM_{\beta} \left( \psi_{\abs{\gamma}} \right)} &\leq C \sum_{\substack{\beta +2 \gamma = \alpha \\ \abs{\gamma} \leq \left( m+1 \right) /2}} \int_{0}^{+ \infty} s^{\abs{\gamma}} {\,} \lVert x^{\beta} f \left( u \left( s \right) \right) \rVert_{1} {\,} ds \\
		&\leq C \sum_{\substack{\beta +2 \gamma = \alpha \\ \abs{\gamma} \leq \left( m+1 \right) /2}} \int_{0}^{+ \infty} s^{\abs{\gamma}} \norm{u \left( s \right)}_{\infty}^{p-1} \lVert x^{\beta} u \left( s \right) \rVert_{1} {\,} ds \\
		&\leq C \left( \norm{\lvert x \rvert^{m+1} {\,} u_{0}}_{1} + \norm{u_{0}}_{1} + \norm{u_{0}}_{\infty} \right)^{p} \int_{0}^{+ \infty} \left( 1+s \right)^{- \frac{n}{2} \left( p-1 \right) + \frac{m+1}{2}} ds \\
		&\leq C \varepsilon^{p} \left( \norm{\lvert x \rvert^{m+1} {\,} \varphi}_{1} + \norm{\varphi}_{1} + \norm{\varphi}_{\infty} \right)^{p},
	\end{align*}
	whence follows
	\begin{align*}
		&\Biggl\lvert \MM_{\alpha} \left( u_{0} \right) + \sum_{\substack{\beta +2 \gamma = \alpha \\ \abs{\gamma} \leq \left( m+1 \right) /2}} \frac{\left( - \nu \right)^{\abs{\gamma}}}{\gamma !} \MM_{\beta} \left( \psi_{\abs{\gamma}} \right) \Biggr\rvert \\
		&\hspace{1cm} \geq \abs{\MM_{\alpha} \left( u_{0} \right)} - \sum_{\substack{\beta +2 \gamma = \alpha \\ \abs{\gamma} \leq \left( m+1 \right) /2}} \frac{\abs{\nu}^{\abs{\gamma}}}{\gamma !} \abs{\MM_{\beta} \left( \psi_{\abs{\gamma}} \right)} \\
		&\hspace{1cm} \geq \varepsilon \abs{\MM_{\alpha} \left( \varphi \right)} -C \varepsilon^{p} \left( \norm{\lvert x \rvert^{m+1} {\,} \varphi}_{1} + \norm{\varphi}_{1} + \norm{\varphi}_{\infty} \right)^{p} \\
		&\hspace{1cm} = \varepsilon \abs{\MM_{\alpha} \left( \varphi \right)} \left( 1- \frac{C \left( \norm{\lvert x \rvert^{m+1} {\,} \varphi}_{1} + \norm{\varphi}_{1} + \norm{\varphi}_{\infty} \right)^{p}}{\abs{\MM_{\alpha} \left( \varphi \right)}} \varepsilon^{p-1} \right).
	\end{align*}
	Thus, by taking account into the assumption that $\MM_{\alpha} \left( \varphi \right) \neq 0$ and setting
	\begin{align*}
		\varepsilon_{\varphi} \coloneqq \min \left\{ \frac{\varepsilon_{0}}{\norm{\varphi}_{1} + \norm{\varphi}_{\infty}}, {\ } \left( \frac{\abs{\MM_{\alpha} \left( \varphi \right)}}{2C \left( \norm{\lvert x \rvert^{m+1} {\,} \varphi}_{1} + \norm{\varphi}_{1} + \norm{\varphi}_{\infty} \right)^{p}} \right)^{\frac{1}{p-1}} \right\},
	\end{align*}
	we obtain
	\begin{align*}
		\Biggl\lvert \MM_{\alpha} \left( u_{0} \right) + \sum_{\substack{\beta +2 \gamma = \alpha \\ \abs{\gamma} \leq \left( m+1 \right) /2}} \frac{\left( - \nu \right)^{\abs{\gamma}}}{\gamma !} \MM_{\beta} \left( \psi_{\abs{\gamma}} \right) \Biggr\rvert \geq \frac{1}{2} \varepsilon \abs{\MM_{\alpha} \left( \varphi \right)} >0
	\end{align*}
	for any $\varepsilon \in \left( 0, \varepsilon_{\varphi} \right]$.
	This in turn implies $\AA_{m+1} \left( 1 \right) - \AA_{m} \left( 1 \right) \not\equiv 0$ by virtue of Theorem \ref{th:P_asymptotics_optimal}.
\end{proof}



\begin{thebibliography}{99} \addcontentsline{toc}{section}{References}
	\bibitem{Aranson-Kramer} 
		I. S. Aranson, L. Kramer,
		\textit{The world of the complex Ginzburg--Landau equation},
		Rev. Modern Phys., \textbf{74} (2002), no. 1, 99--143.
	\bibitem{Cazenave-Dias-Figueira} 
		T. Cazenave, J. P. Dias, M. Figueira,
		\textit{Finite-time blowup for a complex Ginzburg--Landau equation with linear driving},
		J. Evol. Equ., \textbf{14} (2014), no. 2, 403--415.
	\bibitem{Cazenave-Dickstein-Escobedo-Weissler} 
		T. Cazenave, F. Dickstein, M. Escobedo, F. B. Weissler,
		\textit{Self-similar solutions of a nonlinear heat equation},
		J. Math. Sci. Univ. Tokyo, \textbf{8} (2001), no. 3, 501--540.
	\bibitem{Cazenave-Dickstein-Weissler} 
		T. Cazenave, F. Dickstein, F. B. Weissler,
		\textit{Finite-time blowup for a complex Ginzburg--Landau equation},
		SIAM J. Math. Anal., \textbf{45} (2013), no. 1, 244--266.
	\bibitem{Cazenave-Snoussi} 
		T. Cazenave, S. Snoussi,
		\textit{Finite-time blowup for some nonlinear complex Ginzburg--Landau equations},
		Partial differential equations arising from physics and geometry, 172--214,
		London Math. Soc. Lecture Note Ser., 450, Cambridge Univ. Press, Cambridge, 2019.
	\bibitem{Chen-Wang-Wang} 
		J. Chen, B. Wang, Z. Wang,
		\textit{Complex valued semi-linear heat equations in super-critical spaces $E^{s}_{\sigma}$},
		Math. Ann., \textbf{386} (2023), no. 3-4, 1351--1389.
	\bibitem{Chouichi-Majdoub-Tayachi} 
		A. Chouichi, M. Majdoub, S. Tayachi,
		\textit{Global existence and asymptotic behavior of solutions for the complex-valued nonlinear heat equation},
		Ann. Polon. Math., \textbf{121} (2018), no. 2, 99--131.
	\bibitem{Chouichi-Otsmane-Tayachi} 
		A. Chouichi, S. Otsmane, S. Tayachi,
		\textit{Large time behavior of solutions for a complex-valued quadratic heat equation},
		NoDEA Nonlinear Differential Equations Appl., \textbf{22} (2015), no. 5, 1005--1045.
	\bibitem{Clement-Okazawa-Sobajima-Yokota} 
		P. Cl\'{e}ment, N. Okazawa, M. Sobajima, T. Yokota,
		\textit{A simple approach to the Cauchy problem for complex Ginzburg--Landau equations by compactness methods},
		J. Differential Equations, \textbf{253} (2012), no. 4, 1250--1263.
	\bibitem{Duong2019_1} 
		G. K. Duong,
		\textit{Profile for the imaginary part of a blowup solution for a complex-valued semilinear heat equation},
		J. Funct. Anal., \textbf{277} (2019), no. 5, 1531--1579.
	\bibitem{Duong2019_2} 
		G. K. Duong,
		\textit{A blowup solution of a complex semi-linear heat equation with an irrational power},
		J. Differential Equations, \textbf{267} (2019), no. 9, 4975--5048.
	\bibitem{Ginibre-Velo1996} 
		J. Ginibre, G. Velo,
		\textit{The Cauchy problem in local spaces for the complex Ginzburg--Landau equation. I. Compactness methods},
		Phys. D, \textbf{95} (1996), no. 3-4, 191--228.
	\bibitem{Ginibre-Velo1997} 
		J. Ginibre, G. Velo,
		\textit{The Cauchy problem in local spaces for the complex Ginzburg--Landau equation. II. Contraction methods},
		Comm. Math. Phys., \textbf{187} (1997), no. 1, 45--79.
	\bibitem{Gmira-Veron} 
		A. Gmira, L. V\'{e}ron,
		\textit{Large time behaviour of the solutions of a semilinear parabolic equation in $\R^{N}$},
		J. Differential Equations, \textbf{53} (1984), no. 2, 258--276.
	\bibitem{Guo-Ninomiya-Shimojo-Yanagida} 
		J. S. Guo, H. Ninomiya, M. Shimojo, E. Yanagida,
		\textit{Convergence and blow-up of solutions for a complex-valued heat equation with a quadratic nonlinearity},
		Trans. Amer. Math. Soc., \textbf{365} (2013), no. 5, 2447--2467.
	\bibitem{Harada2016} 
		J. Harada,
		\textit{Blowup profile for a complex valued semilinear heat equation},
		J. Funct. Anal., \textbf{270} (2016), no. 11, 4213--4255.
	\bibitem{Harada2017} 
		J. Harada,
		\textit{Nonsimultaneous blowup for a complex valued semilinear heat equation},
		J. Differential Equations, \textbf{263} (2017), no. 8, 4503--4516.
	\bibitem{Hayashi-Kaikina-Naumkin2003_1} 
		N. Hayashi, E. I. Kaikina, P. I. Naumkin,
		\textit{Landau--Ginzburg type equations in the subcritical case},
		Commun. Contemp. Math., \textbf{5} (2003), no. 1, 127--145.
	\bibitem{Hayashi-Kaikina-Naumkin2003_2} 
		N. Hayashi, E. I. Kaikina, P. I. Naumkin,
		\textit{Global existence and time decay of small solutions to the Landau--Ginzburg type equations},
		J. Anal. Math., \textbf{90} (2003), 141--173.
	\bibitem{Ikeda-Sobajima} 
		M. Ikeda, M. Sobajima,
		\textit{Sharp upper bound for lifespan of solutions to some critical semilinear parabolic, dispersive and hyperbolic equations via a test function method},
		Nonlinear Anal., \textbf{182} (2019), 57--74.
	\bibitem{Ishige-Ishiwata-Kawakami} 
		K. Ishige, M. Ishiwata, T. Kawakami,
		\textit{The decay of the solutions for the heat equation with a potential},
		Indiana Univ. Math. J., \textbf{58} (2009), no. 6, 2673--2707.
	\bibitem{Ishige-Kawakami2012} 
		K. Ishige, T. Kawakami,
		\textit{Refined asymptotic profiles for a semilinear heat equation},
		Math. Ann., \textbf{353} (2012), no. 1, 161--192.
	\bibitem{Ishige-Kawakami2013} 
		K. Ishige, T. Kawakami,
		\textit{Asymptotic expansions of solutions of the Cauchy problem for nonlinear parabolic equations},
		J. Anal. Math., \textbf{121} (2013), 317--351.
	\bibitem{Ishige-Kawakami2022} 
		K. Ishige, T. Kawakami,
		\textit{Refined asymptotic expansions of solutions to fractional diffusion equations},
		J. Dynam. Differential Equations, (2022). \href{https://doi.org/10.1007/s10884-022-10224-4}{https://doi.org/10.1007/s10884-022-10224-4}.
	\bibitem{Ishige-Kawakami-Kobayashi} 
		K. Ishige, T. Kawakami, K. Kobayashi,
		\textit{Asymptotics for a nonlinear integral equation with a generalized heat kernel},
		J. Evol. Equ., \textbf{14} (2014), no. 4-5, 749--777.
	\bibitem{Ishige-Kawakami-Michihisa} 
		K. Ishige, T. Kawakami, H. Michihisa,
		\textit{Asymptotic expansions of solutions of fractional diffusion equations},
		SIAM J. Math. Anal., \textbf{49} (2017), no. 3, 2167--2190.
	\bibitem{Jensen} 
		A. Jensen,
		\textit{Commutator methods and a smoothing property of the Schr\"{o}dinger evolution group},
		Math. Z., \textbf{191} (1986), no. 1, 53--59.
	\bibitem{Kawakami} 
		T. Kawakami,
		\textit{Higher order asymptotic expansion for the heat equation with a nonlinear boundary condition},
		Funkcial. Ekvac., \textbf{57} (2014), no. 1, 57--89.
	\bibitem{Kawakami-Takeda} 
		T. Kawakami, H. Takeda,
		\textit{Higher order asymptotic expansions to the solutions for a nonlinear damped wave equation},
		NoDEA Nonlinear Differential Equations Appl., \textbf{23} (2016), no. 5, Art. 54, 30 pp.
	\bibitem{Kawakami-Ueda} 
		T. Kawakami, Y. Ueda,
		\textit{Asymptotic profiles to the solutions for a nonlinear damped wave equation},
		Differential Integral Equations, \textbf{26} (2013), no. 7-8, 781--814.
	\bibitem{Kusaba-Ozawa} 
		R. Kusaba, T. Ozawa,
		\textit{Weighted estimates and large time behavior of small amplitude solutions to the semilinear heat equation},
		Differ. Equ. Appl., \textbf{15} (2023), no. 3, 235--268.
	\bibitem{Machihara-Nakamura} 
		S. Machihara, Y. Nakamura,
		\textit{The inviscid limit for the complex Ginzburg--Landau equation},
		J. Math. Anal. Appl., \textbf{281} (2003), no. 2, 552--564.
	\bibitem{Matsumoto-Tanaka2008} 
		T. Matsumoto, N. Tanaka,
		\textit{Semigroups of locally Lipschitz operators associated with semilinear evolution equations of parabolic type},
		Nonlinear Anal., \textbf{69} (2008), no. 11, 4025--4054.
	\bibitem{Matsumoto-Tanaka2010} 
		T. Matsumoto, N. Tanaka,
		\textit{Well-posedness for the complex Ginzburg--Landau equations},
		Current advances in nonlinear analysis and related topics, 429--442,
		GAKUTO Internat. Ser. Math. Sci. Appl., 32, Gakk\={o}tosho Co., Ltd., Tokyo, 2010.
	\bibitem{Nakamura} 
		M. Nakamura,
		\textit{Remarks on the derivation of several second order partial differential equations from a generalization of the Einstein equations},
		Osaka J. Math., \textbf{57} (2020), no. 2, 305--331.
	\bibitem{Nakamura-Sato} 
		M. Nakamura, Y. Sato,
		\textit{Existence and non-existence of global solutions for the semilinear complex Ginzburg--Landau type equation in homogeneous and isotropic spacetime},
		Kyushu J. Math., \textbf{75} (2021), no. 2, 169--209.
	\bibitem{Nakamura-Takeda} 
		M. Nakamura, H. Takeda,
		\textit{Asymptotic behaviors of global solutions for a semilinear diffusion equation in the de Sitter spacetime},
		Asymptot. Anal., \textbf{125} (2021), no. 3-4, 203--245.
	\bibitem{Nouaili-Zaag} 
		N. Nouaili, H. Zaag,
		\textit{Profile for a simultaneously blowing up solution to a complex valued semilinear heat equation},
		Comm. Partial Differential Equations, \textbf{40} (2015), no. 7, 1197--1217.
	\bibitem{Okazawa} 
		N. Okazawa,
		\textit{Smoothing effect and strong $L^{2}$-wellposedness in the complex Ginzburg--Landau equation},
		Differential equations: inverse and direct problems, 265--288,
		Lect. Notes Pure Appl. Math., 251, Chapman \& Hall/CRC, Boca Raton, FL, 2006.
	\bibitem{Okazawa-Yokota2002} 
		N. Okazawa, T. Yokota,
		\textit{Monotonicity method applied to the complex Ginzburg--Landau and related equations},
		J. Math. Anal. Appl., \textbf{267} (2002), no. 1, 247--263.
	\bibitem{Okazawa-Yokota2010} 
		N. Okazawa, T. Yokota,
		\textit{Subdifferential operator approach to strong wellposedness of the complex Ginzburg--Landau equation},
		Discrete Contin. Dyn. Syst., \textbf{28} (2010), no. 1, 311--341.
	\bibitem{Ogawa-Yokota} 
		T. Ogawa, T. Yokota,
		\textit{Uniqueness and inviscid limits of solutions for the complex Ginzburg--Landau equation in a two-dimensional domain},
		Comm. Math. Phys., \textbf{245} (2004), no. 1, 105--121.
	\bibitem{Ozawa} 
		T. Ozawa,
		\textit{Smoothing effects and dispersion of singularities for the Schr\"{o}dinger evolution group},
		Arch. Rational Mech. Anal., \textbf{110} (1990), no. 2, 165--186.
	\bibitem{Shimotsuma-Yokota-Yoshii2014} 
		D. Shimotsuma, T. Yokota, K. Yoshii,
		\textit{Cauchy problem for the complex Ginzburg--Landau type equation with $L^{p}$-initial data},
		Math. Bohem., \textbf{139} (2014), no. 2, 353--361.
	\bibitem{Shimotsuma-Yokota-Yoshii2016} 
		D. Shimotsuma, T. Yokota, K. Yoshii,
		\textit{Existence and decay estimates of solutions to complex Ginzburg--Landau type equations},
		J. Differential Equations, \textbf{260} (2016), no. 3, 3119--3149.
	\bibitem{Tomidokoro-Yokota} 
		T. Tomidokoro, T. Yokota,
		\textit{Blowup for a complex Ginzburg--Landau equation focusing on the parabolicity},
		Asymptotic analysis for nonlinear dispersive and wave equations, 375--388,
		Adv. Stud. Pure Math., 81, Math. Soc. Japan, Tokyo, 2019.
	\bibitem{Yang} 
		Y. Yang,
		\textit{On the Ginzburg--Landau wave equation},
		Bull. London Math. Soc., \textbf{22} (1990), no. 2, 167--170.
	\bibitem{Yokota-Okazawa2006} 
		T. Yokota, N. Okazawa,
		\textit{Smoothing effect for the complex Ginzburg--Landau equation (general case)},
		Dyn. Contin. Discrete Impuls. Syst. Ser. A Math. Anal., \textbf{13B} (2006), 305--316.
	\bibitem{Yokota-Okazawa2008} 
		T. Yokota, N. Okazawa,
		\textit{The complex Ginzburg--Landau equation (an improvement)},
		Nonlinear phenomena with energy dissipation, 463--475,
		GAKUTO Internat. Ser. Math. Sci. Appl., 29, Gakk\={o}tosho Co., Ltd., Tokyo, 2008.
\end{thebibliography}
\end{document}